\documentclass[ejs]{imsart}

\RequirePackage{amsthm,amsmath,amsfonts, amssymb}
\RequirePackage[numbers]{natbib}
\RequirePackage[colorlinks,citecolor=blue,urlcolor=blue]{hyperref}
\RequirePackage{graphicx}
\usepackage[shortlabels]{enumitem}
\doi{10.1214/154957804100000000}
\pubyear{0000}
\volume{0}
\firstpage{1}
\lastpage{1}
\startlocaldefs

\theoremstyle{plain}
\newtheorem{hyp}{Assumption}[section]
\newtheorem{theorem}{Theorem}[section]

\newtheorem{lem}[theorem]{Lemma}
\newtheorem{corollary}[theorem]{Corollary}
\newtheorem*{hyp*}{Assumption}
\theoremstyle{remark}

\newtheorem{example}[theorem]{Example}
\newtheorem{remark}[theorem]{Remark}
\newtheorem{rem}[theorem]{Remark}

\newcommand{\cc}{{\mathcal C}}

\newcommand{\ck}{{\mathcal K}}
\newcommand{\cK}{{\mathcal K}}

\newcommand{\cq}{{\mathcal Q}}

\newcommand{\Tau}{{\mathrm{Test}}} %\mathcal{T}est

\newcommand{\sparse}{{ s }}

\newcommand{\E}{{\mathbb E}}

\newcommand{\N}{{\mathbb N}}
\renewcommand{\P}{{\mathbb P}}

\newcommand{\R}{{\mathbb R}}

\newcommand{\Z}{{\mathbb Z}}

\newcommand{\rd}{{\rm d}}

\newcommand{\Leb}{{\rm Leb}}

\newcommand{\varmodel}{{\Xi_T}}
\newcommand{\ind}{{\bf 1}}
\newcommand{\scale}{{\overline{\sigma}}}

\newcommand{\inv}[1]{\mathop{\frac{1}{ #1}}\nolimits}
\newcommand{\expp}[1]{\mathop {\mathrm{e}^{ #1}}}

\newcommand{\dT}{\mathfrak{d}_{T}}

\newcommand{\DT}{\mathcal{V}_{T}}
\newcommand{\CT}{C_T}
\newcommand{\tD}{\tilde D}

\newcommand{\kernel}{h}
\newcommand{\Funk}{F}
\newcommand{\Var}{{\rm Var}}
\newcommand{\sgn}{{\rm sgn}}

\newcommand{\coeff}{{\xi}}
\newcommand{\norm}[1]{{\left\lVert #1 \right\rVert}}

\newcommand{\threshold}{t}
\newcommand{\Sep}{\Sigma(\eta,r,s)}

\newlist{propenum}{enumerate}{1} % also creates a counter called 'propenumi'
\setlist[propenum]{label=(\roman*)}

\endlocaldefs

\begin{document}

\begin{frontmatter}

% "Title of the Paper"
\title{Off-the-grid prediction and testing for linear combination of translated features}
\runtitle{Off-the-grid prediction and testing}

% indicate corresponding author with \corref{}
% \author{\fnms{John} \snm{Smith}\thanksref{t1}\corref{}\ead[label=e1]{smith@foo.com}\ead[label=e2,url]{www.foo.com}}
% \thankstext{t1}{Thanks to somebody} 
% \address{line 1\\ line 2\\ \printead{e1}\\ \printead{e2}}

%\author{\fnms{???} \snm{???}\ead[label=e1]{???}}
%\address{\printead{e1}}
%\and
%\author{\fnms{???} \snm{???}\ead[label=e2]{???}}
%\address{\printead{e2}}

\author{\fnms{Cristina}~\snm{Butucea$^1$}\ead[label=e1]{cristina.butucea@ensae.fr}},
\author{\fnms{Jean-François}~\snm{Delmas$^2$}\ead[label=e2]{jean-francois.delmas@enpc.fr}}
\author{\fnms{Anne}~\snm{Dutfoy$^3$}\ead[label=e3]{anne.dutfoy@edf.fr}}
\and
\author{\fnms{Cl\'ement}~\snm{Hardy$^2$}\ead[label=e4]{clement.hardy@enpc.fr}}
%%%%%%%%%%%%%%%%%%%%%%%%%%%%%%%%%%%%%%%%%%%%%%
%% Addresses                                %%
%%%%%%%%%%%%%%%%%%%%%%%%%%%%%%%%%%%%%%%%%%%%%%
\address{$^1$CREST, ENSAE, IP Paris, France, \printead{e1}}

\address{$^2$CERMICS, \'{E}cole des Ponts, France, \printead{e2,e4}}
\address{$^3$EDF R\&D, Palaiseau, France, \printead{e3}}

\runauthor{Butucea, Delmas, Dutfoy, Hardy}

 \begin{abstract}
	We consider a model where a signal (discrete or continuous) is observed with an additive  Gaussian noise process. The signal is issued from a linear combination of a finite but increasing number of translated features. The features are continuously parameterized by their location and depend on some scale parameter. First, we extend previous prediction results for off-the-grid estimators by taking into account here that the scale parameter may vary. The prediction bounds are analogous, but we improve the minimal distance between two consecutive features locations in order to achieve these bounds. 
	
	Next, we propose  a goodness-of-fit test for the  model and give
        non-asymptotic  upper bounds  of  the testing  risk  and of  the
        minimax separation rate between two distinguishable signals.  In
        particular,   our   test   encompasses  the   signal   detection
        framework.  We  deduce  upper  bounds  on  the  minimal  energy,
        expressed as  the $\ell_2$-norm  of the linear  coefficients, to
        successfully detect a  signal in presence of  noise. The general
        model considered in this paper  is a non-linear extension of the
        classical high-dimensional regression model.  It turns out that,
        in this  framework, our  upper bound  on the  minimax separation
        rate matches (up to a logarithmic factor) the lower bound on the
        minimax  separation  rate  for  signal  detection  in  the  high
        dimensional  linear model  associated to  a fixed  dictionary of
        features.   We also  propose  a procedure  to  test whether  the
        features  of  the  observed  signal belong  to  a  given  finite
        collection under the assumption that the linear coefficients may
        vary,  but  have  prescribed  signs under  the  null
        hypothesis. A non-asymptotic upper bound  on the testing risk is
        given.
	
	We illustrate our results on the spikes deconvolution model with Gaussian features on the real line and with the Dirichlet kernel, frequently used in the compressed sensing literature, on the torus.
\end{abstract}
\begin{keyword}[class=MSC2020]
	\kwd[Primary ]{62G05}
\kwd{62G10}
\kwd[; secondary ]{62G08}
\end{keyword}

\begin{keyword}
\kwd{Goodness-of-fit testing}
\kwd{Mixture model}
\kwd{Non-linear regression model}
\kwd{Non-parametric hypotheses testing}
\kwd{Off-the-grid methods}
\kwd{Spikes deconvolution}
\end{keyword}

% history:
% \received{\smonth{1} \syear{0000}}

%\tableofcontents

\end{frontmatter}

% Main text entry area
\section{Introduction}
\label{sec:intro}

In many fields, a signal of interest can be described as a linear combination of shifted source signals having the same shape. Thus, the source signal is supposed to belong to a parametric set of functions (for example, Gaussian, Cauchy or sinusoidal-shaped functions) parameterized by its location parameter. The signal is observed with an additive noise process in discrete or continuous time. We assume that the noise and the observation space can vary with some parameter $T$ increasing with the quality of the observations. 

For example, the chemical analysis of a material is done through spectroscopy and each chemical component is represented by a spiked Gaussian-shaped signal located at some prescribed frequency, see \cite{butucea2021}. The final signal is a linear combination of such spikes. In multiple source detection, sound or image may present a similar structure. 

%This model can be seen as a generalization of a linear regression model (possibly functional regression model) to the case where the features are not observed, but have a parametric shape known up to its location which is a non-linear parameter. 

% This paper is motivated by the study of the spikes deconvolution model \cite{duval2015exact} with applications in spectroscopy (\cite{butucea2021}). In this model, a linear combination (or mixture) of spikes continuously parameterized  is observed with an additive Gaussian noise process. We assume that the spikes are parameterized by a location parameter, that the noise and the observation space can vary with some parameter $T$ increasing with the quality of the observations. 

More general non-linear models (not necessarily location models) for the
features  have been  discussed in  \cite{butucea22}, and  the particular
case   of  location   families   has  been   discussed   in  Section   8
therein.  However, we  allow  here the  features to  depend  on a  scale
parameter which  varies with  $T$. This makes  the proof  technique very
different from the previous one.

We  are interested  in estimating  both the  coefficients of  the linear
combination  and  the  location  parameters of  the  different  features
appearing  in the  signal. We  give  sufficient conditions  in order  to
obtain upper bounds for the quadratic  prediction risk of the same order
as  if  the  non-linear  parameters  were  known.  We  show  that  these
sufficient conditions are milder  than those in \cite{butucea22} without
loosing on the prediction risk bounds.

We  are also  interested in  testing problems.  First, we  want to  test
whether the observations  are issued from a given  linear combination of
features. We remark that it includes the case of signal detection.  This
test problem finds an application in spectroscopy to detect the presence
of a  chemical compound  in a  material. Finally,  we are  interested in
testing whether the observed signal  is a linear combination of features
located at a  prescribed list of values with  linear coefficients having
prescribed  signs under  the null  hypothesis.  This is  of interest  in
spectroscopy:   in   a  material   we   expect   a  list   of   chemical
components. This test problem detects  ageing or important damage to the
material which  can be  detected if  unexpected chemical  components are
present.

\subsection{Model}

Let $T  \in \N$. We  observe a random element  $y$ in the  Hilbert space
$L^2(\lambda_T)$  of square  integrable  functions with  respect to  the
measure  $\lambda_T$   on  the  Borel  $\sigma$-field   of  some  metric
space. The observation is the sum  of a deterministic signal and a noise
process  $w_T$ in  $L^2(\lambda_T)$. We  assume  that the  signal is  an
unknown linear  combination of a  finite unknown number $s$  of features
belonging     to     a      continuously     parameterized     subfamily
$(\varphi_T(\theta), \, \theta \in \Theta)$ of $L^2(\lambda_T)$. We call
this  family  a  continuous  dictionary,   the  weights  of  the  linear
combination  -  the  linear  coefficients, and  the  parameters  of  the
features - the non-linear parameters. Moreover, we assume that the noise
is a Gaussian random process. Thus, the general model is fully specified
by the choice of the Hilbert space of our observation, of the continuous
dictionary of features and of the noise process.

The Hilbert space $L^2(\lambda_T)$ is endowed with the natural scalar
product noted $\left \langle \cdot, \cdot \right
\rangle_{L^2(\lambda_T)} $ and norm $\norm{\cdot}_{L^2(\lambda_T)}$. Let
us define the normalized function $\phi_T$ defined on $\Theta$ by:  
\begin{equation}
\label{eq:def-phi_T}
\phi_{T}(\theta)=\varphi_T(\theta) /\norm{\varphi_T(\theta)}_{L^2(\lambda_T)}.
\end{equation} 
We assume that the signal is  a linear combination with unknown non-zero
linear  coefficients $\beta^{\star}  = (\beta_1^*,\ldots,\beta_s^*)$  in
$(\R^*)^s$ of  an unknown  number  $s \in  \N$ of  active
features     with     unknown     distinct     non-linear     parameters
$\vartheta^\star    =(\theta_1^\star,   \ldots,    \theta_s^\star)   \in
\Theta^s$. 
{We use  the notation $\R^*=\R\backslash  \{0\}$. }

Thus, we observe $y$ in the model:
\begin{equation}
\label{eq:model}
y = \sum_{k=1}^s \beta_k^\star \Phi_T(\theta_k^\star) + w_T \quad  \text{in $L^2(\lambda_T)$}.
\end{equation}
Let us define the multivariate function $\Phi_T$  on $\Theta^s$ by:
\[
\Phi_{T}(\vartheta)= (	\phi_{T}(\theta_{1}), \ldots,
\phi_{T}(\theta_{s}) )^\top 
\quad\text{for}\quad
\vartheta = \left ( \theta_1,\ldots,\theta_s\right ) \in \Theta^s.
\] 
Model \eqref{eq:model} writes 
\[
y = \beta^\star \Phi_T(\vartheta^\star) + w_T \quad  \text{in $L^2(\lambda_T)$}.
\]
When       $s=0$,       we       set      by       convention       that
$\beta^\star \Phi_T(\vartheta^\star)=0$ as well  as $ A^s=\{0\}$ for any
set              $A$.              We             denote              by
$\cq^{\star} = \{\theta_\ell^\star, \, 1 \leq \ell \leq s \}$ the set of
the non-linear  parameters associated  to active features. 

\medskip

In this paper we consider a dictionary given by a one
  dimensional location model scaled with a given $\sigma_T>0$:
\begin{equation}
\label{eq:def-features}
\Big(\varphi_T(\theta) = \kernel(\theta - \cdot,\sigma_T), \, \theta \in \Theta\Big)
\end{equation}
where the  set  $\Theta$ is the  real line $\R$ or
the  torus $\R/\Z$, the real-valued function $\kernel$ is
defined  on $\Theta  \times \mathfrak{S}$,  smooth with  respect to  its
first        variable        and        normalized        so        that
$\norm{h(\cdot,\sigma_T)}_{L^2(\Leb)} =1$, and  $\sigma_T$ is an element
of  the  set  $\mathfrak{S}$  of  admissible  positive  scale  parameter
values. Note  that $\varphi_T$  depends on $T$  only through  the argument
$\sigma_T$.   See  Section  \ref{sec:features-spaces}  for  examples  of
functions  $\kernel$  including  the   Gaussian  scaled-spikes  and  the
low-pass filter.
% Even though the location model considered here is a restriction when compared to general non-linear dictionaries of features considered by e.g. \cite{butucea22}, the scaling $\sigma_T$ introduced here makes this dictionary different. Indeed, this scaling is allowed to depend on $T$ and may improve previous results in the sense that the sufficient conditions on the non-linear parameters in the mixture in order to obtain the prediction and estimation bounds are milder. The least separation distance between the location parameters in this model is allowed to be smaller when compared to unscaled dictionaries, see Remark \ref{rem:separation}.

{ The  process   $y$  is  observed   over  the  support  of   the  measure
$\lambda_T$.  Therefore  it  is  legitimate  to  consider  models  whose
location parameters belong to the smallest interval covering the support
of the measure $\lambda_T$.   Hence, we  introduce the  set $\Theta_T$,  a compact
interval of $\Theta$ (when $\Theta$ is the torus, then we can take
$\Theta_T=\Theta$),  and we shall assume that $\cq^\star$  is a subset
of $\Theta_T$.  We denote by  $|\Theta_T|$ the Euclidean diameter of the
set $\Theta_T$.}

We consider a large variety of Gaussian noise processes. Indeed, we only assume the following mild assumption on $w_T$, where the decay rate  $\Delta_T>0$ controls the noise variance decay as the parameter $T$ grows and $\scale >0$ is the
intrinsic noise level. A wide range of noise processes satisfy our assumptions, see Section~\ref{sec:example}; they can be discrete or continuous, white or coloured under these constraints.
\begin{hyp}[Admissible  noise] \label{hyp:bruit}  Let $T  \in \N$.   The Gaussian
	noise process  $w_T$ satisifies $\mathbb{E}\left[ \norm{w_T}_{L^2(\lambda_T)}^4\right] < + \infty$,  and there exist  a noise
	level  $\scale>0$ and  a decay  rate  $\Delta_T>0$ such  that for  all
	$f\in  L^2(\lambda_T)$,  the  random  variable $\langle  f,w_T  \rangle_{L^2(\lambda_T)}$  is  a
	centered Gaussian random variable satisfying:
	\begin{equation}
  \label{eq:ineq-sharp}
	\Var \left( \langle f,w_T  \rangle_{L^2(\lambda_T)} \right) \leq \scale^2 \,
	\Delta_T\,  \norm{f}_{L^2(\lambda_T)}^2.
	\end{equation}
\end{hyp}

We            assume             that            the            quantity
$\mathbb{E}\left[ \norm{w_T}_{L^2(\lambda_T)}^2\right]$ is known for the
considered models.
Using Cauchy-Schwarz inequality, we get:
\begin{equation}
  \label{eq:ineq-not-sharp}
  \Var \left( \langle f,w_T  \rangle_{L^2(\lambda_T)} \right) \leq
  \E\left[\norm{w_T}^2_{L^2(\lambda_T)}\right]\,  \norm{f}_{L^2(\lambda_T)}^2,
\end{equation}
which is in some examples not as sharp as~\eqref{eq:ineq-sharp}, see
Section~\ref{sec:continuous_noise}. 
We shall  also consider the  finite variance  of the
squared norm of the noise:
\begin{equation}
\label{eq:def:noise-variance}
\varmodel = \Var \left( \norm{w_T}_{L^2(\lambda_T)}^2\right).
\end{equation}

To sum up, the quality of the information provided by our observation $y$ depends on the support of the measure $\lambda_T$ and on the noise $w_T$ through $\Delta_T$. It increases with the parameter $T$.  Due to the particular form of the features, we refer to our model as a Linear combination of translation features (LCTF-model).

In this paper, we are interested both in building estimators $\hat \beta$ and $\hat \vartheta$ of the parameters $\beta^\star$ and $\vartheta^\star$, respectively, and in hypothesis testing problems concerning our model. Our goal is two-fold: on the one hand, we attain best known non asymptotic prediction bounds for the risk measure:
$$
\|\hat \beta \Phi_T(\hat \vartheta) - \beta^\star \Phi_T(\vartheta^\star )\|_{L^2(\lambda_T)}
$$
under less restrictive conditions than previous works. Moreover, we use
the certificate functions designed as tools in these proofs in order to
build test procedures in our model that generalize the signal detection
problem in a linear regression model. On the other hand, we treat the
goodness-of-fit test problem and then, the more general problem of
testing whether the signal in our observation presents only features
included in a prescribed list, {with associated linear coefficients that may vary but cannot change signs.}

\subsection{Previous work} 
Estimating the linear coefficients and the parameters of model \eqref{eq:model} from an observation $y$ has attracted a lot of attention over the past decade. A major contribution in this field comes from the formulation of the BLasso problem in \cite{de2012exact}. This optimization problem on a space of measures allows to estimate both linear coefficients and non-linear parameters without using a grid on the parameter space. This off-the-grid method has successfully  been used in \cite{candes2014towards} and \cite{candes2013super} in the context of super-resolution as well as in \cite{duval2015exact} for spikes deconvolution. High probability bounds for the prediction error have been given in \cite{tang_compressed_sensing}, \cite{tang2014near} and \cite{boyer2017adapting} for the specific dictionary of complex exponential functions continuously parameterized by their frequencies and more recently in \cite{butucea22} for a wide range of dictionaries parameterized over a one-dimensional space. These results are based on certificate functions whose existence have been proven in a very general framework in \cite{poon2018geometry} provided that the non-linear parameters of the mixture are well-separated with respect to a Riemannian metric.

\medskip

Goodness-of-fit tests are used to check whether observations are indeed derived from a given statistical model. We refer to the monograph \cite{ingsterbook2003} for a comprehensive presentation of goodness-of-fit testing. When we consider a finite dictionary of features $(\varphi_T(\theta), \theta \in \cq)$ with $\cq$ a known finite subset of $\Theta$, the model \eqref{eq:model} can be rewritten as a linear regression model, possibly of high dimension depending on the size of the finite dictionary $p := \operatorname{Card}(\cq)$.  In this case, testing the goodness-of-fit of the model amounts to testing whether the linear coefficients in the mixture are equal to some given linear coefficients. When the dictionary is known, the testing problem is homogeneous in the linear coefficients $\beta$ and is therefore equivalent to testing $\beta \equiv 0$, which is a signal detection problem.  

Signal detection has raised a lot of interest over the past decades. It is well known that the alternative hypothesis $H_1$ (presence of signal) must be well separated from the null hypothsesis $H_0$ (only noise) in order to have tests with small risks. The separation can be seen as a minimal signal intensity allowing the detection. Then, it is a matter of interest to evaluate the minimax separation rate, i.e., the smallest separation that  allows to distinguish the tested hypotheses.  
In \cite{Ermakov90},  asymptotic rates for the minimax separation in the framework of signal detection are derived for the non-parametric Gaussian white noise model. Non-asymptotic rates were then derived in 
\cite{Baraud02} and later in  \cite{Laurent12} to tackle the case of heterogeneous variances. 
We refer to the monograph \cite{GineNickl} for an overview of non-parametric hypotheses testing. Regarding the high dimensional regression model where the observation is of dimension $T$ and the dictionary is fixed, known and of size $p$, the work of \cite{Ingster2010} established the following asymptotic minimax separation rates under coherence assumptions on the dictionary:
$$
\frac{1}{T^{\frac{1}{4}}} \wedge \sqrt{\frac{s}{T}\log (p)} \wedge \frac{p^{\frac{1}{4}}}{\sqrt{T}} \cdot
$$
The signal intensity is expressed by the $\ell_2$-norm of the linear coefficients. Their lower bounds on the asymptotic minimax separation stand for both fixed and random designs whereas their upper bounds stand for random designs. The work of \cite{Arias-Castro11} does not tackle the high dimension but provides tests achieving the minimax separation for fixed designs under coherence assumptions on the dictionary. We note that the existing results do not apply to our context. 

%In this paper we shall  consider that our features come from a continuous dictionary and have unknown location parameters. Hence, 
For the non-linear extension of linear regression models that we consider here, goodness-of-fit testing does not reduce to signal detection as the mixture is not homogeneous with respect to the non-linear parameters. Therefore, we introduce new testing procedures. We stress that one of the test statistics is not derived from estimators of the linear coefficients. In fact, depending on the sparsity of the signal, the dimension of the observation and the size of the dictionary, plug-in methods using sparse estimators might not be the best way to proceed. They do not always lead to the minimal separation. In this sense, testing is a very different statistical problem from estimation.

\subsection{Description of the results} 

{The   aim  of   this  paper}   is   twofold.  First,   we  improve   on
\cite{butucea22}  in  the  case  of  linear  combination  of  translated
features  by  giving  bounds  on   the  prediction  error  under  milder
separation  constraints between  the {unknown} non-linear
parameters in $\cq^\star$. Indeed,  the sufficient separation conditions
between two neighboring non-linear  parameters 
are difficult to track explicitly. In all generality, they can be rather
restrictive  and scale  with  a factor  $s$  for arbitrary  dictionaries
satisfying  the conditions.  In the  particular case  of Gaussian-shaped
features,  more  explicit  calculations  are possible  and  the  minimal
separation reduces to some constant value.

In this paper, due to the shape of our dictionary of features, {\it i.e.} a location model scaled by some $\sigma_T$, we get more explicit sufficient separation conditions which are less restrictive. This is achieved by taking the scale parameter of the features $\sigma_T$ into account. In particular, in the case of Gaussian-shaped features, the minimal separation is of order $\sigma_T$. Intuitively, this is can be explained by the fact that for peaked features (with small scaling parameter $\sigma_T$) we may distinguish spikes located at smaller (by a factor $\sigma_T$) distance.

The second goal of this paper is to study hypotheses  testing problems in  these models.   We give
procedures  for the  goodness-of-fit of  the mixture  model in  order to
determine          whether          the          unknown          signal
$\beta^\star  \Phi_T(\vartheta^\star)$ is  equal to  a reference  signal
$\beta^0\Phi_T(\vartheta^0)$      for       some      known      vectors
$\beta^0      \in       (\R ^*)^{s^0}$      and
$\vartheta^0 \in  \Theta_T^{s^0}$. Under  our assumptions, the  model is
identifiable, thus  the null  hypothesis is  equivalent to  testing that
$\beta^\star, \vartheta^\star$  coincide with $\beta^0,  \vartheta^0$ up
to  a permutation.   This setup  includes the  case of  signal detection
where the null  hypothesis is $\beta^\star \equiv 0$, that  is $s=0$ On this aspect, our minimal intensity rates allowing signal detection are similar up to a log factor  to the rates obtained in \cite{Ingster2010} for high dimensional linear models. We
propose a combined procedure based  on differences between the reference
signal $\beta^0 \Phi_T(\vartheta^0)$ and either the observation $y$ or a
reconstructed   signal   obtained   from   estimators   of   the   model
parameters. In  order to successfully  perform the test, we  remove from
the alternative hypothesis the signals  whose proximity to the reference
signal $\beta^0\Phi_T(\vartheta^0)$ is  below some separation parameter,
with  respect to  the  norm $\norm{\cdot}_{L^2(\lambda_T)}$.  We give  a
non-asymptotic upper bound of the testing risk and deduce an upper bound
on  the   minimal  separation   needed  to  distinguish   two  different
signals.  This upper  bound yields  two  regimes according  to the  test
procedures that  we define and study.  In the case of  signal detection,
the  separation can  be expressed  as  the $\ell_2$-norm  of the  linear
coefficients  of   the  observed   mixture.  In  particular,   when  the
observation $y$ is  issued from a non-linear extension  of the classical
high-dimensional  regression  model,  our  upper bound  matches  (up  to
logarithmic  factors)   the  asymptotic  lower  bound   of  the  minimal
separation  needed  to  distinguish  two signals  that  are  mixture  of
features from a finite high-dimensional dictionary.

Moreover, we test  the presence of at most $s_0$  prescribed features in
the mixture with  arbitrary linear coefficients of given  sign. That is,
we  test  whether  for  each  $\epsilon\in \{+,  -\}$  the  unknown  set
$\cq^{\star,   \epsilon}=   \{\theta^\star_k\in   \cq^\star\,   \colon\,
\epsilon \beta^\star_k>0  \}$ is a  subset of $\cq^{0,  \epsilon}$, with
$\cq^{0,  +}$  and  $  \cq^{0,-}$   being  given disjoint  finite  subsets  of
$\Theta_T$.  This  setup is issued  from an application  to spectroscopy
(see  \cite{butucea2021}),   where  the   presence  of   other  chemical
components than the prescribed ones are indicating ageing or substantial
modifications of the analyzed material.  To separate the null hypothesis
from the  alternative hypothesis, we  introduce a discrepancy that  is 0
{if      and      only      if}      the      parameters
$(\beta^\star, \vartheta^\star)$ belong to the null hypothesis.  We give
an upper  bound on  the minimal separation  to successfully  perform our
test.  The test  statistic introduced and studied in  this context makes
explicit  use   of  the   construction of certificates   used  in compressed sensing  \cite{candes2011probabilistic,tang_compressed_sensing,poon2018geometry}, super resolution  \cite{candes2014towards}, spikes deconvolution \cite{duval2015exact},    as well as in \cite{butucea22,tang2014near,boyer2017adapting}  for
establishing    the   prediction    rates   of    the   estimators    of
$(\beta^\star,  \vartheta^\star)$.  We stress  the  fact  that the  test
statistic is not an estimator  of the discrepancy measure separating the
null  and  the  alternative  hypotheses,  as  is  usually  the  case  in
non-parametric tests.

\subsection{Roadmap of the paper} Section~\ref{sec:exmaple_models} gives
several  possible  specific choices  in  our  general model  by  showing
examples  of dictionaries  of  features, of  observation  spaces and  of
Gaussian processes (white or coloured under our assumptions). In Section
\ref{sec:assumptions}, we start by  presenting the assumptions needed to
perform a successful estimation of  the linear coefficients and location
parameters  of our  model. After  giving a  prediction bound  in Theorem
\ref{maintheorem}, we  show in Lemma \ref{lem:identifiability}  that the
required assumptions  are sufficient conditions for  the identifiability
of the model. In Section  \ref{sec:goodness-of-fit}, we test whether the
observation  derives from  a given  mixture or  from some  other mixture
sufficiently   separated  from   the   latter.  We   give  in   Theorems
\ref{th:non_sparse}  and  \ref{th:higly_sparse}  bounds of  the  testing
risks  associated  to   two  different  test  procedures.   We  show  in
Corollaries \ref{th:non_sparse_R} and \ref{th:higly_sparse_R} that these
two tests give two regimes for our upper bound on the minimal separation
to distinguish two different signals from an observation contaminated by
noise.  We  also provide  a discussion  on the  comparison of  our upper
bounds    with    some    existing    lower    bounds.     In    Section
\ref{sec:goodness-of-fit-features},  we  propose  a  procedure  to  test
whether the  active features in  the observed  signal belong to  a given
finite  collection with  linear coefficients  of prescribed  signs. Both
hypotheses of this  test problem are composite and a  new measure of the
separation between  these hypotheses  has been introduced.  The proposed
test relies on  the certificates used in the  proof of the prediction
bounds in  an original  way. A  bound of  the testing  risk is  given in
Theorem       \ref{theorem:inclusion}       and       in       Corollary
\ref{theorem:inclusion_R},  we  provide  an  upper  bound  on  the  minimax
separation rate.   The examples of Gaussian  scaled spikes deconvolution
on  $\R$ and  low-pass  filter  on $\R  /\Z$  are  adressed in  Sections
\ref{ex:gaussian-spike} and  \ref{sec:low-pass-filter}. Some  proofs can
be found in Section~\ref{sec:proofs}.

\section{Specific models covered by our general model}
\label{sec:exmaple_models}
We consider a large variety of models: discrete models where the process $y = (y(t_1),\ldots,y(t_T))$ is observed on a finite grid $t_1<\ldots < t_T$ or continuous models where the process $y= y(t)$ is observed on a continuous interval.

\subsection{Examples of feature functions}
\label{sec:features-spaces}

Various continuous dictionaries of features can be considered under regularity conditions required later on. They include many parametric families of functions known in statistics and compressed sensing literature.
\begin{enumerate}
	\item \textit{Gaussian scaled-spikes deconvolution.} The noisy linear combination of translated and re-scaled Gaussian features corresponds to:
	\begin{equation}
	\label{eq:def:gaussian-features}
	\kernel(t,\sigma) \mapsto  \frac{ \exp(-t^2/2\sigma^2) }{\pi^{1/4} \sigma ^{1/2} }
	\quad\text{on $\Theta \times \mathfrak{S} = \R \times \R_+^*$}.
	\end{equation} 
	The example of Gaussian spikes deconvolution is analyzed in full details in \cite[Section 8]{butucea22} when $\sigma_T$ does not depend on $T$. We shall consider here that the scale parameter $\sigma_T$ may vary with $T$.
	\item \textit{Multi-resolution approximation.}  We consider
	the  normalized Shannon  scaling function:
	\[
	\kernel(t,\sigma) \mapsto  \sqrt{\sigma}\,  \frac{\sin(\pi t/\sigma)}{\pi t}
	\quad\text{on $\Theta \times \mathfrak{S} = \R \times \R_+^*$}.
	\]
	The associated dictionary allows to recover functions whose Fourier transform have their support in $[-\pi/\sigma, \pi / \sigma]$ (see \cite[Theorem 3.5]{mallat2008wavelet}).

	\item \textit{Low-pass filter.}
	We consider the normalized Dirichlet kernel on the torus for some cut-off frequency $f_c \in \N^*$ and $T=2f_c+1$: 
	\begin{equation}
	\label{eq:low-pass-h}
	\kernel(t,\sigma) = \frac{1}{\sqrt{T}}\sum_{k=-f_c}^{f_c} \expp{2 i\pi k t} = \frac{\sin(T \pi  t)}{\sqrt{T} \, \sin(\pi t)},
	\end{equation}
	with  $\sigma = 1/T$, $T \in 2 \N^* +1$ and $t\in \Theta  = \R / \Z $.
	The example of the low-pass filter is adressed in \cite{duval2015exact}, where exact support recovery results are obtained for the BLasso estimators. This dictionary is also used in \cite{candes2013super} in the context of super-resolution. Bounds on some prediction risks (different from those considered in this paper) are established therein for estimators obtained by solving the constrained formulation of the BLasso.
\end{enumerate}

\subsection{Examples of observation spaces and Gaussian noise processes}
\label{sec:example}

We consider both discrete-time and continuous-time processes in our general model.

\subsubsection{Discrete-time process observed on a regular grid}
\label{sec:discrete_noise}
Consider  a  real-valued process  $y$  observed  over a  regular  grid
$t_1<\ldots< t_T$ of a symmetric interval $[-a_T, a_T] \subset \R$, with $T\geq
1$,  $t_j= -a_T + j \Delta_T $ for $j=1, \ldots, T$ and grid step:
$\Delta_T=2 a_T/{T}$. We set:
\begin{equation}
\label{eq:def-lambdaT-reg-grid}
\lambda_T = \Delta_T\sum_{j=1}^T \delta_{t_j}
\end{equation}
Then, we see $y$ as an element of $L^2(\lambda_T)$.
We have for any function $f \in L^2(\lambda_T)$ that
$\norm{f}_{L^2(\lambda_T)}=\sqrt{\Delta_T}  \norm{f}_{\ell_2}$, where  the right-hand  side is
understood     as     the     $\ell_2$-norm (Euclidean norm)      of     the     vector
$(f(t_1), \ldots, f(t_T))$.

We assume that $(a_T, T\geq 2)$  is a sequence of positive numbers, such
that:   $\lim_{T\rightarrow    \infty   }   a_T   =    +   \infty$   and
$\lim_{T\rightarrow  \infty }  \Delta_T =  0$  so that  the sequence  of
measures $(\lambda_T,  T \geq  1)$ converges with  respect to  the vague
topology  towards the  Lebesgue measure,  noted $\Leb$,  on $\R$.   When
$\Theta=\R$,  it  is  therefore  natural   in  this  case,  to  consider
non-linear parameters  within the support  of the observations  and take
$\Theta_T = [-a_T,a_T]$. When $T$ tends  to infinity, in the limit model
the  observation  corresponds  to  a square  integrable  random  process
indexed on $\Theta = \R$.  In the  case of periodic signals, we may take
the sets  $\Theta$ and  $\Theta_T$ to be  the torus $\R  / \Z$,  and the
limit  measure is  then the  Haar measure  identified with  the Lebesgue
measure.

In this formalism, the  noise $w_T\in L^2(\lambda_T)$ is
given by: 
\begin{equation}
\label{eq;def-wT-reg-grid}
w_T(t)=\sum_{j=1}^T G_j \ind_{\{t_j\}}(t),
\end{equation}
where $\ind_A$ 
denotes the indicator function of an arbitrary set $A$ and
$(G_1,\cdots,G_T)$ is a centered Gaussian random vector with independent
entries of variance $\scale^2$.

In this case Assumption~\ref{hyp:bruit} 
 holds with an equality in~\eqref{eq:ineq-sharp} and
 $\E[\norm{w_T}^4_{L^2(\lambda_T)}]$
 % =\scale^4 \Delta_T^2 T(T+2)$
 is finite. 
{Notice  that $\E[\norm{w_T}^2_{L^2(\lambda_T)}]=\scale^2 \Delta_T
\, T$, thus the Cauchy-Schwarz inequality~\eqref{eq:ineq-not-sharp} gives an upper bound larger by a factor $T$ than the value given by~\eqref{eq:ineq-sharp}.}
We also have  that $\varmodel = 2 \scale^4 \Delta_T^2 \, T$.

Finally, the model writes:
	$$
	y_j := y\left( t_j \right) = \sum_{k=1}^s \beta^\star_k \, \phi_T \left(\theta_k^\star,  t_j \right) + G_j,\quad j=1,\ldots, T.
	$$
	We stress that when the noises $(G_j)_{1 \leq j \leq T}$ are independent the model encompasses the Gaussian sequence model where the mean vector is the sampling of  a linear combination of shifts of a known function.

%$\mathbb{E}[\norm{w_T}^4_{L^2(\lambda_T)}]=\scale^4 \Delta_T^2 T(T+2)$ 

\subsubsection{Continuous-time processes}
\label{sec:continuous_noise}
Assume we  observe a real-valued process  $y$ on  a topological  state
space.  We note $\lambda=\lambda_T$ for a $\sigma$-finite measure on the
state   space.    In   this   framework,    $y$   is   an   element   of
$L^2(\lambda)$.     Let    us     assume    that     the    noise     is
$w_T=\sum_{k\in    \N}   \sqrt{\xi_k}    \,    G_k\,   \psi_k$,    where
$(G_k, k\in \N)$ are independent centered Gaussian random variables with
variance  $\scale^2$,  $\psi=(\psi_k,  k  \in   \N)$  an  orthonormal
sequence   of
$L^2(\lambda)$, and   $\xi=(\xi_k, k\in \N)$  a
summable sequence of non-negative real numbers. The sequences $\psi$ and
$\xi$  may depend  on $T$.  Let $\norm{\xi}_{\ell_p}$  denote the  usual
$\ell_p$-norm of the sequence $\xi$.  We have:
\[
\Var( \langle f , w_T\rangle_{L^2(\lambda)})=\scale^2 \sum_{k\in \N} \xi_k\,
\langle f, \psi_k \rangle_{L^2(\lambda)}^2  
\leq \scale^2 \,\Delta_T\,   \norm{f}^2_{L^2(\lambda)},
\]
with $\Delta_T=  \norm{\xi}_{\ell_\infty }=\sup_{k \in \N} \xi_k$. 
We also have $\mathbb{E}[\norm{w_T}^2_{L^2(\lambda)}] =\scale^2
\norm{\xi}_{\ell_1}$
 % $\mathbb{E}[\norm{w_T}^4_{L^2(\lambda)}] =\scale^4\left(
 %   \norm{\xi}_{\ell_1}^2+ 2   \norm{\xi}_{\ell_2}^2
 % \right)$, 
 and
 $\varmodel = \Var(\norm{w_T}_{L^2(\lambda)}^2) = 2 \scale^4
 \norm{\xi}_{\ell_2}^2$. 
 In particular Assumption~\ref{hyp:bruit} holds.
 
\medskip
 
% We remark that Assumption~\ref{hyp:bruit} holds as 
% $\mathbb{E}[\norm{w_T}^4_T]=3\scale^4 \sum_{k\in  \N} \xi_k^2 + \scale^4\sum_{k,\ell\in  \N, k \neq \ell}\xi_k \xi_\ell  $ is finite
% and $\Var(\norm{w_T}_{L^2(\lambda_T)}^2) = 2 \scale^4 \sum_{k \in \N} \xi_k^2$. Moreover, we have with $\Delta_T=  \sup_{k \in \N} \xi_k$:
% \[
% \Var( \langle f , w_T\rangle_{L^2(\lambda_T)})=\scale^2 \sum_{k\in \N} \xi_k\,
% \langle f, \psi_k \rangle_{L^2(\lambda_T)}^2  
% \leq \scale^2 \,\Delta_T\,   \norm{f}^2_{L^2(\lambda_T)}.
% \]
% In this example the noise $w_T$ depends on the parameter $T$ only if $\xi$, and thus $\Delta_T$, depend on $T$.

We may consider different choices for $\xi$ that lead to different
values for $\varmodel$, the variance of the squared norm of  the
noise. For instance, our framework encompasses the truncated white noise
by taking for all $k \in \N$, $\xi_k = T^{-1}\ind_{ \{1 \leq k \leq T
  \}}$. In this case, we have
$\norm{\xi}_{\ell_\infty }=1/T$ and $\norm{\xi}_{\ell_1}=1$. In
particular, we get that the inequality~\eqref{eq:ineq-not-sharp} is not
as sharp as~\eqref{eq:ineq-sharp}
since $\Delta_T=1/T$ whereas 
$\mathbb{E}[\norm{w_T}^2_{L^2(\lambda)}] =\scale^2$.

% elementary calculations give

% $\Delta_T=1/T$ and $\varmodel = 2\scale^4 /T$. 
% Note that, for centered noise process $w_T$ we may write the bound
% $$
% \Var( \langle f , w_T\rangle_{L^2(\lambda_T)}) = \mathbb{E} (\langle f , w_T\rangle^2_{L^2(\lambda_T)}) \leq  \norm{f}^2_{L^2(\lambda_T)} \cdot \left(\mathbb{E}[\norm{w_T}^4_T] \right)^{1/2}.
% $$
% However, $\mathbb{E}[\norm{w_T}^4_T] = 3\scale^4 \frac 1T +  \scale^4\left(1 - \frac 1T \right)$ is of constant order and thus leading to a larger upper bound than $\Delta_T$ when $T$ grows large.

\section{Assumptions and prediction bounds}
\label{sec:assumptions}
We recall in this section assumptions and definitions from Sections 3-5 of \cite{butucea22} in a simpler way adapted to our framework. In \cite{butucea22}, the authors established high probability bounds for prediction and estimation errors associated to some estimators of $\beta^\star$ and $\vartheta^\star$ tackling a wider range of dictionaries.

\subsection{Regularity of the features}
We gather in this section the hypotheses that will be required  on the features defined by \eqref{eq:def-features}.

Recall that  the parameter  space $\Theta$  is either  $\R$ or the   torus 
$\R  /  \Z$   endowed  with  the  Lebesgue  measure
$\Leb$.  For convenience,  we write  $|x-y|$ for  the  Euclidean  distance
between $x$ and $y$  either on  $\R$ or   on  the
torus. Recall also that $L^2(\lambda_T)$ and $L^2(\Leb)$ are the sets of
square integrable  functions on  $\Theta$ with  respect to  the measures
$\lambda_T$ and $\Leb$ respectively. We denote $\mathfrak{S}$ the set of
scale parameter values.
\begin{hyp}[Smoothness of the features]
	\label{hyp:reg-f} 
	Let $\kernel$ be a function defined on $\Theta\times \mathfrak{S}$. Let $T \in \N$ and $\sigma_T \in \mathfrak{S}$. We assume that the function $\theta \mapsto h(\theta,\sigma_T)$  is 
	of class    $\cc^3$ on $\Theta$.  We assume furthermore that
        $\norm{\kernel(\cdot,\sigma_T)}_{L^2(\Leb)}=1$, and that for all
        $\theta \in \Theta$ $\norm{\kernel
          (\theta-\cdot,\sigma_T)}_{L^2(\lambda_T)}    >   0$  and all
        $ i \in \{0,\cdots,3\}$: 
        \[
          \norm{\partial_\theta^{i}\kernel(\cdot,\sigma_T)}_{L^2(\Leb)}
          < + \infty
          \quad\text{and}\quad
          \norm{\partial_\theta^{i} \kernel  (\theta-\cdot,
            \sigma_T)}_{L^2(\lambda_T)} < + \infty.
        \]
\end{hyp}

	% $\norm{\kernel (\theta-\cdot,\sigma_T)}_{L^2(\lambda_T)}    >   0$ 
	% $\norm{\partial_\theta^{i}\kernel(\cdot,\sigma_T)}_{L^2(\Leb)} < + \infty$  and $\norm{\partial_\theta^{i} \kernel  (\theta-\cdot, \sigma_T)}_{L^2(\lambda_T)} < + \infty$.

Recall the function $\varphi_T$ defined by \eqref{eq:def-features} and notice that Assumption \ref{hyp:reg-f} implies  $\norm{\varphi_T (\theta)}_{L^2(\lambda_T)} > 0$ on $\Theta$. We define the function:
\begin{equation}
\label{def:g_T}
g_T(\theta)= \norm{\partial_\theta
	\phi_T(\theta)}_{L^2(\lambda_T)}^2, \quad \text{ where } \phi_{T}(\theta) = \varphi_T (\theta) /  \norm{\varphi_T (\theta)}_{L^2(\lambda_T)}.
\end{equation}

\begin{hyp}[Positivity of $g_T$]
	\label{hyp:g>0}
	Assumption \ref{hyp:reg-f} holds and we have	$g_T>0$ on $ \Theta$.
\end{hyp}
Let us mention that if for all $\theta \in \Theta$, $\varphi_T (\theta)$
and $\partial_\theta \varphi_T (\theta)$ are linearly independent
functions of $L^2(\lambda_T)$ and $\norm{\partial_\theta \varphi_T
  (\theta)}_{L^2(\lambda_T)} > 0$, then  $g_T>0$ on  $\Theta$ (see \cite[Lemma 3.1]{butucea22}).

\subsection{Definition of the kernel and its approximation}

\subsubsection{Measuring the colinearity of the features}

We define the symmetric kernel $\cK_T$ on $\Theta^2$ by: 
\begin{equation}
\label{eq:def-KT}
\cK_T(\theta, \theta')=\langle \phi_{T}(\theta), \phi_{T}(\theta')
\rangle_{L^2(\lambda_T)}.
\end{equation}
The kernel $\ck_T$ measures the colinearity of two features belonging to the continuous dictionary.
It does not \emph{a priori} have a simple form. In the following, we  approximate this kernel by another kernel easier to handle.

As mentioned  in the introduction, we  consider in this paper  a setting
where the sequence of measures $(\lambda_T, T \geq 1)$ converges in some
sense   towards   the   Lebesgue   measure  $\Leb$   on   $\Theta$.   In
\cite{butucea22}, the  kernel $\cK_T$  was free  of any  scale parameter
$\sigma_T$  and  authors  have   considered  a  pointwise  limit  kernel
$\cK_\infty = \lim_{T\to \infty} \cK_T$ which  is free of $T$ and allows
to continue the proofs under some assumptions. However, due to our scale
parameter $\sigma_T$ which  decreases towards zero with $T$,  we show in
the following example that the pointwise limit kernel is degenerate.

\begin{example}[Degenerate limit kernel]
  Consider    the   discrete-time    process   presented    in   Section
  \ref{sec:discrete_noise}     with      the     measure     $\lambda_T$
  from~\eqref{eq:def-lambdaT-reg-grid}   and   the   Gaussian   features
  \eqref{eq:def:gaussian-features}              from             Section
  \ref{sec:features-spaces} scaled by the sequence $(\sigma_T, T\geq 1)$
  that  tends  towards   zero  when  $T$  grows  to   infinity  so  that
  $\lim_{T \rightarrow + \infty} \Delta_T  /\sigma_T = 0$. In this case,
  the  sequence  of measures  $(\lambda_T,  T  \geq 1)$  converges  with
  respect to the  vague topology towards the Lebesgue measure  and it is
  easy to  check that  $\cK_\infty$, the pointwise  limit of  the kernel
  $\cK_T$, is equal to zero almost everywhere and to 1 on the diagonal.
	
\end{example}

Thus, instead of the pointwise limit kernel $\cK_\infty$, we shall approximate (for finite large enough $T$) the kernel $\cK_T$ by a kernel $ \ck_T^{\text{prox}}$ of the form:
\begin{equation}
\label{eq:def_K-app}
\ck_T^{\text{prox}} :(\theta,\theta') \mapsto \Funk(|\theta-\theta'|/\sigma_T),
\end{equation}
where  $\Funk$ is  a  real-valued  function  defined  on $\R_+$  with
$\Funk(0)=1$.  (Recall that $|\theta-\theta'|$ is the Euclidean distance
between $\theta$ and $\theta'$ on $\Theta$ which is either $\R$ or the
torus $\R/ \Z$.)
% Since $\Funk$  is even,  notice that  if it  is of  class
% $\cc^{2\ell}$ then $\ck_T^{\text{prox}}$ is of class $\cc^{\ell, \ell}$.}
Notice that if $\Funk$ is  of class $\cc^{2\ell}$ with $F^{(2i+1)}(0)=0$
for $i\in \{0, \ldots,  \ell -1\}$ for some integer $\ell \geq 1$ 
(which is the case  if $\Funk$ can be
extended into  an even function of  class $\cc^{2\ell}$ on $\R$  ), then
$\ck_T^{\text{prox}}$ is  of class  $\cc^{\ell, \ell}$.  The  choice of
the  function $\Funk$  follows  from  the model  given  by $\kernel$  in
\eqref{eq:def-features}, so  that $\ck_T$ and  $\ck_T^{\text{prox}}$ are
close      (see     $\ref{hyp:proximity-setting}$      of     Assumption
\ref{hyp-estimation}     below).       We     refer      to     Sections
\ref{ex:gaussian-spike} and \ref{sec:low-pass-filter}  for examples with
$h$ given by \eqref{eq:def:gaussian-features} and \eqref{eq:low-pass-h},
respectively.  The  introduction of the kernel  $\ck_T^{\text{prox}}$ is
significantly   different   from    the   approximation   developed   in
\cite{butucea22}.

\subsubsection{Covariant derivatives of the kernel}
Let $\cK$ be a symmetric kernel of class
$\cc^2$ such  that the  function $g_\cK$  defined on
$\Theta$ by:
\begin{equation}
\label{eq:def-gK}
g_\cK(\theta)= \partial^2_{x,y}
\cK( \theta,\theta),
\end{equation}
is positive, where $\partial_x$
(respectively $\partial_y$) denotes the usual  derivative with respect to the
first (respectively second) variable. Under Assumptions \ref{hyp:reg-f} and $\ref{hyp:g>0}$, the definitions \eqref{def:g_T}  and  \eqref{eq:def-gK} coincide so that $g_T = g_{\ck_T}$ on $\Theta$.

Similarly to \cite{poon2018geometry},   we    introduce   the   covariant
derivatives
which reduce to elementary expressions since  the location parameters are
one-dimensional. More precisely following \cite[Section 4]{butucea22}, we set for  a smooth function $f$ defined on $\Theta$,  $\tilde D_{0;\cK}[f] =  f$, $\tilde D_{1;\cK}[f] =  g_\ck^{-1/2} f'$ and  for $i \geq 2$:
\begin{equation*}
\label{eq:tDi}
\tD_{i; \cK}[f]=\tD_{1; \cK}[\tD_{i-1; \cK}[f]].
\end{equation*}
Let us assume that the kernel $\ck$ has the form $\ck(\theta,\theta') = \left \langle f(\theta), f(\theta') \right \rangle_{L^2(\lambda)} $ for some function $f$ of class $\mathcal{C}^3$ and some measure $\lambda$ on $\Theta$.
We then define the covariant derivatives (see (27) in \cite{butucea22})
of $\cK$ for $i, j\in \{0, \ldots, 3\}$ and
$\theta, \theta'\in \Theta$ by:
\begin{equation*}
\label{def:derivatives_kernel}
\cK^{[i,j]}(\theta,\theta') = \langle
\tD_{i;\cK}[f](\theta), \tD_{j;\cK}[f](\theta')
\rangle_{L^2(\lambda)}. 
\end{equation*}
We also define the function $h_\cK$ on $\Theta$ by:
\begin{equation*}
\label{eq:def-h_K}
h_\cK(\theta)=\cK^{[3,3]}(\theta, \theta).
\end{equation*}

The previous notation will be used both for the kernel $\cK_T$ in \eqref{eq:def-KT}, which is determined by the particular choice of the features, but also for the kernel $\cK_T^{\text{prox}}$ in \eqref{eq:def_K-app}. The latter is determined by the  function $\Funk$ and we derive next the particular expressions of $g_{\cK_T^{\text{prox}}}$ and of the covariant derivatives of $\cK_T^{\text{prox}}$ under additional assumptions on $F$.

For a real valued function $f$ defined on a set $A$, we write $\norm{f}_{\infty} \!\!= \sup_{x \in A} |f(x)|$.
\begin{hyp}[Properties of the function $\Funk$]
	\label{hyp:properties_k}
Let $\Funk$ be a function defined on $\R_+$ of class
$\mathcal{C}^6$ with $F(0)=1$ and $F^{(2i+1)}(0)=0$ for $i\in \{0,1,
2\}$. We set: 
 \begin{equation}
\label{eq:def-g}
g_\infty = -F''(0).
\end{equation}
	We assume that:
	\begin{equation}
	\label{def:M_h_M_g}
	\begin{aligned}
	&g_\infty > 0, \quad  L_6 := g_\infty^{-3}|\Funk^{(6)}(0)| < + \infty,
	\quad\\
	&\text{and}\quad  L_i := g_\infty^{-i/2} \,   \norm{\Funk^{(i)}}_{\infty}
	<+\infty
	\quad\text{for all  $ i \in \{0, \cdots,4\}$}.
	\end{aligned}
	\end{equation}
\end{hyp}
We give  the covariant  derivatives of the  kernel $\ck_T^{\text{prox}}$
according  to  the  definition  given  in  \cite[(27)]{butucea22}:
for any $\theta,\theta' \in \Theta$
and $i,j \in \{0,\cdots,3\}$,
% This
% definition    coincides    with   \eqref{def:derivatives_kernel}    when
% $\ck_T^{\text{prox}}(\theta,\theta')   =    \left   \langle   f(\theta),
%   f(\theta')  \right  \rangle_{L^2(\lambda)}$  on  $\Theta^2$  for  some
% smooth  function  $f$  and  some  measure  $\lambda$  on  $\Theta$,  see
% \cite[Lemma 4.3]{butucea22}.
\begin{equation}
\label{eq:cov-derivative-F}
\ck_T^{\text{prox}[i,j]}(\theta,\theta') =  \frac{(-1)^j}{g_\infty^{(i+j)/2}} \Funk^{(i+j)} \left (|\theta-\theta'|/{\sigma_T} \right ).
\end{equation}
We notice that we have for any $\theta \in \Theta$:
\begin{equation}
\label{eq:def-g-inf}
g_{\ck_T^{\text{prox}}}(\theta) = g_\infty / \sigma_T^2.
\end{equation}

\subsubsection{Measuring the quality of the approximation}
\label{sec:approximation}
In this section, we  quantify the proximity of the kernel $\ck_T$ and $\ck_T^{\text{prox}}$. 

Following \cite{poon2018geometry}, we define the  one-dimensional Riemannian  metric $\dT(\theta,\theta')$   between
$\theta, \theta'\in \Theta$ by:
\begin{equation}
\label{eq:def_metric}
\dT(\theta,\theta')=  |G_T(\theta) -G_T(\theta')|,
\end{equation}
where  $G_T$ is a primitive of the function $\sqrt{g_T}$  assumed positive on $\Theta$ thanks to Assumption \ref{hyp:g>0}.

Recall that $\Theta_T$, introduced below the model \eqref{eq:model}, is a compact sub-interval of $\Theta$. 
Since $\Theta_T$ is compact, under
Assumptions~\ref{hyp:g>0} and~\ref{hyp:properties_k}, we deduce that the constant $\CT$
below is positive and finite, where:
\begin{equation}
\label{eq:def-rho}
\CT=\max \left(\sup_{\Theta_T} \sqrt{\frac{ g_{\ck_T^{\text{prox}}}}{ g_T
}},\sup_{\Theta_T} \sqrt{\frac{ g_{T} }{ g_{\ck_T^{\text{prox}}} }} \right). 
\end{equation}
Elementary calculations show that the metric $\dT$ defined in \eqref{eq:def_metric} is equivalent, up to a factor $\sigma_T$, to the Euclidean metric on $\Theta_T$ as for any $\theta,\theta' \in \Theta_T$:
\begin{equation}
\label{eq:equi-dT-dI}
\inv{\CT}\,  \sqrt{g_\infty} \, \sigma_T^{-1}\,  | \theta- \theta'| \leq  \dT(\theta,\theta') \leq  \CT\,  \sqrt{g_\infty} \, \sigma_T^{-1}\,  | \theta- \theta'|.
\end{equation}

\medskip

In order to quantify  the approximation of $\cK_T$ by
$\cK_T^{\text{prox}}$, we set:
\begin{equation}
\label{def:V_1}
\begin{aligned}
\DT=\max( \DT^{(1)}, \DT^{(2)})
\end{aligned}
\end{equation}
\begin{equation*}
	\quad\text{with}\quad \!\!
	\DT^{(1)}= \!\!\max_{i,j\in \{0, 1, 2\} }\, \sup_{\Theta_T^2} |
	\cK_T^{[i,j]} - \cK_T^{{\text{prox}}[i,j]}|
	\quad\text{and} \!\! \quad
	\DT^{(2)}=\sup_{\Theta_T} |h_{\cK_{T}} - h_{\cK_T^{\text{prox}}}|.
\end{equation*}

\subsection{Boundedness and local concavity on the diagonal  of the approximating kernel}

Recall the definition of the kernel $\ck_T^{\text{prox}}$ given by
\eqref{eq:def_K-app} using the  function $\Funk$.
We quantify the boundedness and local concavity on the diagonal of the kernel $\ck_T^{\text{prox}}$ using for $r> 0$:
\begin{align}
\label{eq:def-e0}
\varepsilon(r)
& = 1 - \sup \left\{ |F(r')|; \quad r' \geq r
\right\},\\
\label{eq:def-e2}
\nu(r)
&= -\sup \left\{
F''(r')/g_\infty; \quad r' \in [0, r] \right\}.
\end{align}

We also quantify the colinearity between $s \in \N$ features belonging to the continuous dictionary, by setting for $u>0$:
\begin{multline}
\label{eq:def-delta-rs}
\delta(u,s) = \inf \Big\{ \delta>0\, \colon\,   \max_{1\leq  \ell \leq
	s} \sum\limits_{k=1, k\neq
	\ell}^{s}
g_\infty^{-\frac{i}{2}}|F^{(i)}(x_\ell - x_k)|
\leq u, \\  \text{ for all } i \in  \{0,1,2, 3 \} \text{ and } 
(x_1,\cdots,x_s)  \in \R^s(\delta) \Big\},
\end{multline}
where for any subset $A$ of $\R$ or $\R/\Z$  and for any $\delta \geq 0$,
\begin{equation}
\label{eq:def-set-restriction}
A^s(\delta)= \Big  \{  (\theta_1,\cdots,\theta_s) \in A^s\,
\colon\,  |\theta_{\ell} - \theta_{k}| >  \delta \text{  for
	all distinct } k, \ell\in \{1, \ldots, s\}   \Big \}. 
\end{equation}
with the conventions $\inf  \emptyset=+\infty $, and for $s=0, 1$: $A^0(\delta) = \{0\}$  and $A^1(\delta) = A$.

Following \cite{butucea22},  we   define     quantities    which    depend    only on  the function $\Funk$ and on a real parameter $r>0$:
\begin{align*}
H_{\infty}^{(1)}(r)
&= \inv{2} \wedge L_{2} \wedge L_{3} \wedge L_4 \wedge L_6 \wedge
\frac{\nu(2 r)}{10} \wedge
\frac{\varepsilon(r/2) }{10},\\ 
H_{\infty}^{(2)} (r)
&= \frac{1}{6} \wedge \frac{8\varepsilon (r/2) }{10(5 + 2
	L_{1})}  \wedge 
\frac{8\nu(2 r)}{9 ( 2L_{2} + 2 L_{3} + 4)},
\end{align*}
where the constants $L_i$ are defined in \eqref{def:M_h_M_g}.

\subsection{Main assumption and identifiability of the model}
\label{sec:assumption}

We summarize here all assumptions that are needed for the following results. They concern the features, the function $F$ characterizing the proxy kernel $\cK_T^{\text{prox}}$, the proximity of the kernel $\cK_T$ defined by the original features to the prox kernel $\cK_T^{\text{prox}}$ and, last but not least, the assumption that two neighbouring non-linear parameters $\theta$ and $\theta'$ are at least separated by some constant multiplied by $\sigma_T$. This is the most important improvement on the sufficient conditions in \cite{butucea22}, as the scaling parameter $\sigma_T$ can be chosen small in some models.
\begin{hyp}
	\label{hyp-estimation}
	Let $T \in \N$, $s \in \N$,  $r \in \left (0,1/\sqrt{2 \, g_\infty \, L_{2}} \right )$, $\eta \in (0,1)$ and a subset $\cq \subset \Theta_T$ of cardinal $s$. 
	
	\begin{propenum}		
		\item\textbf{Regularity           of            the           dictionary
			$\varphi_T$:}\label{hyp:theorem_properties_dico}    The   dictionary
		function   $\varphi_T$   satisfies   the  smoothness   conditions   of
		Assumption~\ref{hyp:reg-f}.    The    function   $g_T$    defined   in
		\eqref{def:g_T},    satisfies    the     positivity    condition    of
		Assumption~\ref{hyp:g>0}.
		
		\item\textbf{Properties            of            the            function
			$\Funk$:}         \label{hyp:theorem_properties_k}        Assumption
		\ref{hyp:properties_k} holds  and we  have $\varepsilon(r/2) >  0$ and
		$\nu(2 r) > 0$.
		
		\item\textbf{Proximity             to              the             limit
			setting:}       \label{hyp:proximity-setting}       The       kernel
		$\cK_T$      defined      from      the      dictionary,
		see~\eqref{eq:def-KT}, is sufficiently close  to the kernel $\cK_T^{\text{prox}}$  in
		the sense that we have:
		\[
		\CT \leq 2
		\]
		and if $s \geq 1$, we have in addition:
		\[
		\DT\leq H_{\infty}^{(1)}(r)
		\quad\text{and}\quad  (s-1) \DT\leq (1-\eta) H_{\infty}^{(2)}(r).
		\]
		
		\item\textbf{Separation of the non-linear parameters:}  \label{hyp:separation} If $s\geq1$, we have:
		\[ 
		\delta(\eta H_{\infty}^{(2)}(r),s)  < + \infty \quad \text{ and for any
		}  \theta \neq \theta' \in \cq, \quad |\theta - \theta' |> \sigma_T \,
		\Sep,  
		\]
		where,
		\[
		\Sep = 4 \, \max \left (r g_\infty^{-1/2}, 2  \,\delta(\eta
		H_{\infty}^{(2)}(r),s) \right ).
		\]
	\end{propenum}
\end{hyp}
\begin{rem}[On the separation condition]
	{The separation condition corresponds to the minimal distance between any pair of nonlinear parameters ensuring that a coherence function remains bounded from above by a specified constant dependent on the dictionary. This condition is mathematically represented in \eqref{eq:def-delta-rs} and expressed with the following coherence function:
		$$
		\max_{1\leq \ell \leq s} \sum_{k=1, k \neq \ell}^{s} g_\infty^{-\frac{i}{2}} |F^{(i)}(x_\ell - x_k)|,
		$$
		where $\{x_1,\cdots,x_s\}$ is a set of nonlinear parameters.
	This function is quite similar to the Babel function introduced in \cite{tropp2004greed}, which measures the maximum total coherence between a fixed atom and a collection of other atoms in a finite dictionary. In linear cases (when the dictionary consists of a finite number of atoms), keeping the Babel function below a certain threshold  allows for the derivation of results on the recovery of sparse signals. 
	We stress that similar separation conditions to   Assumption \ref{hyp-estimation} are common in super-resolution, compressed sensing and spikes deconvolution for recovering signals  derived from  continuous dictionaries, see \cite{candes2014towards,duval2015exact,poon2018geometry} among many other references.}
\end{rem}

In Sections~\ref{ex:gaussian-spike} and~\ref{sec:low-pass-filter} we give simplified expressions of the quantities involved in the previous assumption for the particular models in hand.

Under  Assumption   \ref{hyp-estimation},  we  shall   build  consistent
estimators  for   $\beta^\star$  and  $\vartheta^\star$  of   the  model
\eqref{eq:model}  and test  statistics.   The following  lemma gives  an
identifiability  result  for the  considered  model  under the  previous
assumptions. Its  proof relies on
the construction of  certificates from \cite{butucea22} and  is based on
ideas  developed  in  \cite{de2012exact}  for  exact  reconstruction  of
measures,  see  Lemma  1.1  therein.    We  recall  that  by  convention
$\beta^\star\Phi_T(\vartheta^\star)=0$ when $s=0$.
\begin{lem}[Sufficient conditions for identifiability]
	\label{lem:identifiability}
	Let $T \in \N$ and let  $r \in \left (0,1/\sqrt{2 \, g_\infty \, L_{2}} \right )$, $\eta \in (0,1)$. Suppose that Assumption \ref{hyp-estimation} holds for the set $\cq^\star = \{\theta^\star_1,\cdots,\theta_{s}^\star\}  \subset \Theta_T$ of cardinal $s \in \N$ and for the set $\cq^0  = \{\theta^0_1,\cdots,\theta_{s^0}^0\} \subset \Theta_T$  of cardinal $s^0 \in \N$.
	Then, for any vectors $\beta^\star \in (\R^*)^s, \beta^0 \in (\R^*)^{s^0}$, we have that, up to the same permutation on the components of $\beta^\star$ and $\vartheta^\star$:
	\begin{equation}
	\label{eq:identifiability}
	\beta^\star \Phi_T(\vartheta^\star) = \beta^0 \Phi_T(\vartheta^0) \quad \! \! \text{in }L^2(\lambda_T), \, \, \, \text{ implies that }\,\,\,  s=s^0, \,\, \beta^\star = \beta^0, \, \, \, \vartheta^\star= \vartheta^0.
	\end{equation}
\end{lem}
The proof is in Section~\ref{sec:proof_Lemma_2_4}. 
\begin{remark}
	Recall that if $s\geq 1$, then  $\beta^\star$ is a $s$-dimensional vector with non-zero entries. Under the assumptions of Lemma \ref{lem:identifiability} we have that:
	\[
	\beta^\star \Phi_T(\vartheta^\star) = 0 \quad \text{if and only if} \quad s =0.
	\]  
\end{remark}

\subsection{Prediction error bound}

We define the estimators $\hat \beta$ and $\hat \vartheta$ of $\beta^\star$ and $\vartheta^\star$ as the solution to the following regularized optimization problem with a real tuning parameter $\kappa>0$ and a bound $K$ on the unknown number $s$ of active features in the observed mixture:
\begin{equation}
\label{eq:generalized_lasso}
(\hat{\beta},\hat{\vartheta}) \in \underset{\beta \in
	\mathbb{R}^{K}, \vartheta \in \Theta_{T}^K}{\text{argmin}} \quad
\frac{1}{2}\norm{y - \beta\Phi_{T}(\vartheta)}_{L^2(\lambda_T)}^{2} +\kappa
\norm{\beta}_{\ell_1},
\end{equation}
where $\norm{\cdot}_{\ell_1}$ corresponds to the usual $\ell_1$ norm.
Since the interval $\Theta_T$ on  which the  optimization of the  non-linear
parameters is performed is  a compact interval and the function $\Phi_T$ is continuous, the existence of at least a  solution is guaranteed. 
The  bound $K$ on the number $s$ of features in the mixture from model \eqref{eq:model} allows to formulate an optimization problem. It can be arbitrarily large. In particular, it is not involved in the bounds on estimation and prediction risks given in \cite{butucea22} with high probability (see Remark 2.4 therein). We stress that the constants in \cite{butucea22} appearing in those bounds may \emph{a priori} depend on $T$ when the features are scaled by $\sigma_T$. We show below that, in fact, those bounds still hold with constants free of $T$. The results in \cite{butucea22} as well as the proof of Theorem \ref{maintheorem} below rely on the existence of certificate functions. In \cite{butucea22}, sufficient conditions for the certificate functions to exist are given, see Proposition 7.4 and 7.5 therein. Those conditions require the non-linear parameters in $\cq^\star$ to satisfy the separation condition \eqref{eq:separation-Q-star}. In our framework where the scaling $\sigma_T$ decreases to zero, it turns out that this separation is in general increasing with $s$ and decreasing with $T$. However, for some dictionary composed of translated spikes that vanish quickly, it converges to zero when both $s$ and $T$ grow to infinity. We refer to Section \ref{ex:gaussian-spike} in this direction.

\medskip

Recall the definitions of $g_\infty$ and $L_2$ given by \eqref{eq:def-g}  and \eqref{def:M_h_M_g}. The following theorem is a variation of \cite[Theorem 2.1]{butucea22}. 

\begin{theorem}
	\label{maintheorem}
	Let $T \in \N, s \in \N^*$, $K \in \N^*$, $\eta \in (0,1)$,  $r \in \left (0,1/\sqrt{2 \, g_\infty \, L_{2}} \right )$. Assume we observe the random element $y$ of $L^2(\lambda_T)$ under the regression model (\ref{eq:model}) with unknown parameters  $\beta^\star \in (\R^{*})^s$ and $\vartheta^\star= \left (
	\theta_1^\star,\cdots,\theta_s^\star\right )$ a vector with distinct entries in
	$ \Theta_T$, a compact interval of $\Theta$,   such that Assumption \ref{hyp-estimation} holds for $\cq^\star = \{\theta_1^\star,\cdots,\theta_s^\star\} \subset \Theta_T$. Assume that the unknown number of active features $s$ is bounded by $K$. Suppose also that the noise process $w_T$ satisfies Assumption \ref{hyp:bruit} for a noise level $\scale>0$ and a decay rate for the noise variance $\Delta_T>0$.

	Then, there exist finite positive constants $\mathcal{C}_i$, for $i=0,  \ldots, 3$, 
	%	$\mathcal{C}_1$, $\mathcal{C}_2$, $\mathcal{C}_3$,
	depending on the function $\Funk$ and on $r$ such that for 
	any $\tau > 1$ and a tuning parameter:
	\begin{equation}
	\label{eq:bound-kappa}
	\kappa \geq \mathcal{C}_1 \scale \sqrt{\Delta_T  \log (\tau)},
	\end{equation}
	we have the prediction error bound of the estimators $\hat{\beta}$ and $\hat{\vartheta}$ defined in
	\eqref{eq:generalized_lasso} given by:
	\begin{equation}
	\label{eq:main_theorem}
	\begin{aligned}
	\norm{\hat{\beta}\Phi_{T}(\hat{\vartheta}) -
		\beta^{\star}\Phi_{T}(\vartheta^{\star}) }_{L^2(\lambda_T)}
	&\leq  \mathcal{C}_0 \,  \sqrt{\sparse} \, \kappa,
	\end{aligned}
	\end{equation}
	with  probability larger than $1  -
	\mathcal{C}_2 \left (  \frac{|\Theta_T|}{ \sigma_T  \, \tau \sqrt{\log (\tau)} }\vee
	\frac{1}{\tau}\right )$ where $|\Theta_T|$ is the Euclidean length of $\Theta_T$. Moreover, with the same probability, the difference of the $\ell_1$-norms of $\hat{\beta}$ and $\beta^\star$ is bounded by:
	\begin{equation}
	\label{eq:main_theorem_diff_l1}
	\left |\| \hat{\beta}\|_{\ell_1} - \| \beta^\star \|_{\ell_1} \right | \leq \mathcal{C}_3 \, \kappa \, \sparse.
	\end{equation}
\end{theorem}
\begin{proof}
	The proof is similar to the proof of \cite[Theorem~2.1]{butucea22}
	where one replaces the limit kernel noted $\ck_\infty$ therein by the approximating kernel
	$\ck_T^{\text{prox}}$ defined in~\eqref{eq:def_K-app}. The main difference is in checking condition 
	$(v)$ in Theorem~2.1 on the existence of certificate functions. This is done by using 
	Propositions 7.4 and 7.5 therein, and by noticing that the special form of the approximating kernel 
	$\ck_T^{\text{prox}}$ implies that the constants involved do not depend on the 
	scale parameter $\sigma_T$. Indeed Equation~\eqref{eq:cov-derivative-F} clearly entails  that they 
	do not depend on the scale parameter. The details of the proof  are left to the interested reader.
\end{proof}

\medskip
\begin{remark}[On the separation]
	\label{rem:separation}
	We perform the estimation of $\beta^\star$ and $\vartheta^\star = (\theta_1^\star,\cdots,\theta_s^\star)$ from model \eqref{eq:model} under the separation condition:
	\begin{equation}
	\label{eq:separation-Q-star}
	|\theta_k^\star - \theta_\ell^\star| \geq \sigma_T \,  \Sep, \quad \text{ for all } 1 \leq k , \ell \leq s, \, k \neq \ell,
	\end{equation}
	with $\Sep$ given in $\ref{hyp:separation}$ of Assumption \ref{hyp-estimation}. Taking into account the separation condition, the number of admissible features which can be used for the prediction is at most of order $|\Theta_T|/\sigma_T$; this provides a natural upper bound on $s$.  
	As $\eta$ is usually fixed, we highlight that the least separation bound tends towards zero when the scaling $\sigma_T$ goes down to zero.
\end{remark}
%Notice that even if the constants $\mathcal{C}_i$, for $i=0, \ldots, 3$, depend only on the function $\Funk$ and  on $r$, Assumption~\ref{hyp-estimation}~$\ref{hyp:proximity-setting}$ implies that  $F$ is chosen according to  the function $\kernel$. The estimation risks on $\beta^\star$ and  $\vartheta^\star$ can be further deduced as in \cite[Equations~(9-10)]{butucea22}. 

%\medskip

%\begin{remark}	Notice that $\beta^\star \Phi_T(\vartheta^\star) = \beta^0 \Phi_T(\vartheta^0)$ can be re-written as $\tilde \beta \Phi_T(\tilde \vartheta ) = 0$ for some $(\tilde \beta,\tilde \vartheta) \in \R^{\tilde s} \times \Theta_T^{\tilde s}$ where the components of $\tilde \vartheta$ are the elements of $\cq^\star \cup \cq^0$, $\tilde s = \operatorname{Card}(\cq^\star \cup \cq^0)$ and the entries of $\tilde \beta$ are up to a sign those of $\beta^\star$ or $\beta^0$. In fact, one could show Lemma \ref{lem:identifiability}  by supposing that Assumption  \ref{hyp-estimation} stands for the set $\cq^\star \cup \cq^0$. However, as Assumption \ref{hyp-estimation} requires pairwise separations between the considered location parameters (see $\ref{hyp:separation}$ of Assumption \ref{hyp-estimation}), we remark that this condition would be much stronger than requiring that the sets $\cq^\star$ and $\cq^0$ verify Assumption \ref{hyp-estimation} separately. 
%\end{remark}

\section{Goodness-of-fit for the LCTF model}
\label{sec:goodness-of-fit}
In this section, we build a test procedure to decide if the observation $y$ derives from a given linear combination of translated features.
We build a test $\Psi$, {\it i.e.} a measurable function of the observation $y$ taking value in $\{0,1\}$, in order to distinguish  a null hypothesis $H_0$ against an alternative $H_1(\rho)$ depending on a nonnegative separation parameter $\rho$. We recall that the maximal type I and II error probabilities are $\sup_{(\beta^\star, \vartheta^\star) \in H_0}\E_{(\beta^\star,\vartheta^\star)}[\Psi]$ and $\sup_{(\beta^\star, \vartheta^\star) \in H_1(\rho)} \E_{(\beta^\star,\vartheta^\star)}[1-\Psi]$, respectively, 
where $\Psi$ is a function of $y$ which is equal to  $\beta^\star\Phi_T(\vartheta^\star) + w_T$ under $\E_{(\beta^\star,\vartheta^\star)}$. The maximal testing risk is the sum of the former quantities, that is:
\[
R_\rho(\Psi) = \sup_{(\beta^\star, \vartheta^\star) \in H_0}\E_{(\beta^\star,\vartheta^\star)}[\Psi] + \sup_{(\beta^\star, \vartheta^\star) \in H_1(\rho)} \E_{(\beta^\star,\vartheta^\star)}[1-\Psi],
\]
and the minimax testing risk is: 
\begin{equation}
\label{def:minimax_testing_risk}
R^\star_\rho = \inf_{\Psi}R_\rho(\Psi),
\end{equation}
where the infinimum is taken over all the measurable functions from $L^2(\lambda_T)$ to $\{0,1\}$.
The minimax separation rate of the test problem is defined for any $\alpha \in (0,1)$ as: 
\begin{equation}
\label{def:minimax_separation}
\rho^\star(\alpha) = \inf \{\rho>0: R^\star_{\rho} \leq \alpha\}.
\end{equation}

\subsection{Test problem}
Let $s^0 \in \N$ and consider the set  $\Theta_T^{s^0}(\delta^0)\subset \Theta_T^{s^0}$ of vectors whose components are pairwise separated by a distance $\delta^0 \geq 0$ (recall the definition \eqref{eq:def-set-restriction}). Consider the vectors $\beta^0 \in (\R^*)^{s^0}$ and $\vartheta^0 = (\theta^0_1, \cdots, \theta_{s^0}^0) \in \Theta_T^{s^0}(\delta^0)$. By convention, we have for $s^0=0$ that $\beta^0 = 0$, $\vartheta^0=0$ and $\beta^0\Phi_T(\vartheta^0) = 0$.

We  build a test procedure based on the observation $y$ to decide, for some $\delta^\star \geq 0$, whether:
\begin{equation}
\label{eq:H0_goodness_of_fit}
\begin{cases}
H_0 : &(\beta^\star, \vartheta^\star) \in (\R^*)^s \times \Theta_T^{s}(\delta^\star) \!\quad \text{s.t.} \!\quad   \quad \beta^\star \Phi_T(\vartheta^\star) = \beta^0 \Phi_T(\vartheta^0), \\
H_1 (\rho): \!\!\!&(\beta^\star, \vartheta^\star) \in (\R^*)^s \times \Theta_T^{s}(\delta^\star) \quad \!\text{s.t} \!\quad \norm{\beta^\star \Phi_T(\vartheta^\star) - \beta^0 \Phi_T(\vartheta^0)}_{L^2(\lambda_T)} \geq \rho,
\end{cases}
\end{equation}
where $\rho$ is a nonnegative  separation parameter.
When  Assumption \ref{hyp-estimation} holds for the sets $\cq^\star = \{\theta^\star_1,\cdots,\theta_{s}^\star\}$ and  $\cq^0  = \{\theta^0_1,\cdots,\theta_{s^0}^0\}$,  by Lemma  \ref{lem:identifiability}, the null hypothesis implies that $ (\beta^\star,\vartheta^\star)=(\beta^0, \vartheta^0) $ (up to the same permutation on the components of $\beta^\star$ and $\vartheta^\star$). We remark that the separation condition from Point $\ref{hyp:separation}$ of Assumption~\ref{hyp-estimation} required between the elements of $\cq^\star$ (resp. $\cq^0$) is automatically satisfied when  $\delta^\star \geq \sigma_T\, \Sep$ (resp. $\delta^0 \geq \sigma_T\, \Sigma(\eta,r,s^0)$).

We shall denote the distribution under the null hypothesis as associated to the parameters $(\beta^0,\vartheta^0)$ and see that the maximal type I error probability writes in this case  $\E_{(\beta^0,\vartheta^0)}[\Psi]$ for $\E_{(\beta^\star,\vartheta^\star)}[\Psi]$. 
Furthermore, when $s^0= 0$, under Assumption \ref{hyp-estimation} for the set $\cq^\star$, Lemma \ref{lem:identifiability} implies that the null hypothesis reduces to $H_0: s= 0$.

\subsection{Main results}
We consider the test procedure  $\Psi_\Tau(\threshold)$ associated to a real valued statistic $\Tau$ (measurable function of the observation $y$) and a threshold $\threshold >0$ (defining a critical region) given by: 
\begin{equation}
\label{eq:def-test}
\Psi_\Tau(\threshold) = \ind_{ \{|\Tau| >\threshold \}}.
\end{equation}
We recall that for a test $\Psi$, we accept $H_0$ when $\Psi = 0$ and reject it when $\Psi = 1$.

{It is  now well-known that several  test statistics may
  be combined to cover for several regimes in the set of parameters. Our
  test statistics will  be produced by estimating in  two different ways
  $\norm{\beta^\star       \Phi_T(\vartheta^\star)       -       \beta^0
    \Phi_T(\vartheta^0)}_{L^2(\lambda_T)}^2$, the squared $L^2(\lambda_T)$
  distance separating  the null and  the alternative hypothesis.  On the
  one hand,  we plug-in  the estimators from  the previous  section into
  this distance and, on the other  hand, we use the observed process $y$
  as a proxy  for the unknown signal,  in which case it  is necessary to
  remove            the            known            bias            term
  $\mathbb{E}  \left   [  \norm{w_T}^2_{L^2(\lambda_T)}  \right   ]$  as
  follows.}

Let $s^0 \in \N$ and consider known linear coefficients and location parameters $\beta^0 \in (\R^*)^{s^0}$ and $\vartheta^0 = (\theta^0_1, \cdots, \theta_{s^0}^0) \in \Theta_T^{s^0}$, respectively.
We define two statistics $\Tau_1$ and $\Tau_2$ by:
\begin{equation}
\label{def:test_statistic}
\begin{aligned}
&\Tau_1 = \norm{y-\beta^0\Phi_T(\vartheta^0)}_{L^2(\lambda_T)}^2 - \mathbb{E} \left [ \norm{w_T}^2_{L^2(\lambda_T)} \right ] ,\\
&\Tau_2 = \norm{\hat{\beta}\Phi_T(\hat{\vartheta}) - \beta^0\Phi_T(\vartheta^0)}_{L^2(\lambda_T)}^2,
\end{aligned}
\end{equation}  
where $\hat{\beta}$ and $\hat{\vartheta}$ denote the estimators obtained from \eqref{eq:generalized_lasso} for a given value of the tuning parameter $\kappa$ and a bound $K$ on the unknown number $s \in \N$ of active features in the observed signal.

Recall the definition \eqref{eq:def:noise-variance} of $\varmodel$, the
variance of the squared $L^2(\lambda_T)$-norm of  the noise $w_T$. The
following theorem gives an upper bound of the maximal testing risk
associated to the test $\Psi_{\Tau_1}(\threshold)$ for some positive
threshold $\threshold$ and positive separation $\rho$. Its proof can be
found in Section~\ref{sec:proof_thm_3_1}.

\begin{theorem}
	\label{th:non_sparse}
	Let $T \in \N$ and $s^0 \in \N$. Let:
        \[
          \delta^\star \geq 0 \quad\text{and}\quad
          \delta^0 \geq 0.
        \]
        Assume that we observe the random element $y$ of $L^2(\lambda_T)$ under the regression model \eqref{eq:model} with unknown parameters   $s \in \N$, $\beta^\star \in (\R^*)^s$ and $\vartheta^\star  \in \Theta_T^s(\delta^\star)$. Let $\beta^0 \in (\R^*)^{s^0}$ and $\vartheta^0 \in \Theta_T^{s^0}(\delta^0)$. Suppose that Assumption \ref{hyp:reg-f} on the smoothness of the features holds. Suppose that Assumption \ref{hyp:bruit} holds for a noise level $\scale>0$ and a decay rate for the noise variance $\Delta_T>0$. 
	
	Then, the test $\Psi_{\Tau_1}$ in  \eqref{eq:def-test}  using $\Tau_1$ in \eqref{def:test_statistic} satisfies:
	\begin{equation}
	\label{eq:non_sparse}
	R_\rho\left(\Psi_{\Tau_1}(\threshold)\right) \leq  \frac{\varmodel}{\threshold^2} +  \frac{4 \, \varmodel}{(\rho^2 - \threshold)^2} +  \expp{ - ( \rho^2 - \threshold)^2/(32\scale^2 \Delta_T \rho^2)},
	\end{equation}
	for any threshold $\threshold$ and any separation $\rho$  such that $ \rho^2 > \threshold > 0$.
\end{theorem}

%%%%%%%%%%%%%%%%%%%%%%%%%%%%%%%%%%%%

We deduce from Theorem \ref{th:non_sparse} upper bounds on the minimax separation $\rho^\star$ defined in \eqref{def:minimax_separation} for the goodness-of-fit test problem \eqref{eq:H0_goodness_of_fit}.
\begin{corollary}
	\label{th:non_sparse_R}
	Under the framework and the assumptions of Theorem \ref{th:non_sparse}, the minimax separation rate for the test problem \eqref{eq:H0_goodness_of_fit} verifies for any $\alpha \in (0,1)$:
	\begin{equation}
	\label{eq:non_sparse_R}
	\rho^\star(\alpha) \leq  \rho^{(1)}(\alpha) 
	\quad\text{with}\quad 
	\rho^{(1)}(\alpha):=\max \left( \left (\frac{40 \varmodel}{\alpha} \right )^{1/4} , 8 \, \scale \sqrt{ 2\Delta_T \log\left (\frac{2}{\alpha} \right )} \right ).
	\end{equation}
\end{corollary}

\begin{proof}[Proof of Corollary \ref{th:non_sparse_R}]
	This result is a direct consequence of Theorem \ref{th:non_sparse} by taking  the threshold $\threshold$ of the test therein equal to $\rho^2/ 2$.
	Then, we have that for $\rho > 0$:
	\[
	R_\rho^\star  \leq R_\rho \left (\Psi_{\Tau_1}(\rho^2/2) \right )\! \leq \! \frac{4\varmodel}{\rho^4} + \frac{16 \, \varmodel}{\rho^4} + \expp{ - \rho^2/(128\scale^2 \Delta_T)}=  \frac{20 \, \varmodel}{\rho^4} + \expp{ - \rho^2/(128\scale^2 \Delta_T)}.
	\]
	We deduce that $R_\rho^\star \leq \alpha$ for any $\alpha \in (0,1)$ whenever the separation $\rho$ satisfies: 
	\begin{equation}
	\label{rho:non_sparse_R}
	\rho \geq \left(\frac{40\varmodel}{\alpha} \right )^{\frac{1}{4}} \vee \scale\sqrt{128 \, \Delta_T \log\left (\frac{2}{\alpha} \right )}.
	\end{equation}
	This implies \eqref{eq:non_sparse_R}.
\end{proof}

In the following theorem, we give a bound of the maximal testing risk associated to the test $\Psi_{\Tau_2}(\threshold)$ using $\Tau_2$ in \eqref{def:test_statistic} for solving the test problem \eqref{eq:H0_goodness_of_fit}. The statistic $\Tau_2$ is defined using estimators of the model parameters $(\beta^\star,\vartheta^\star)$. In view of recovering the latter, we assume that the minimal distance $\delta^\star$ (resp. $\delta^0$) is large enough so that Point $\ref{hyp:separation}$ of Assumption \ref{hyp-estimation} is satisfied for the components of $\vartheta^\star$ (resp. $\vartheta^0$).

Recall the definitions of $g_\infty$ and $L_2$ given by \eqref{eq:def-g}  and \eqref{def:M_h_M_g}, that $|\Theta_T|$ denotes the Euclidean length of the compact set $\Theta_T$ and  $\Sigma$ defined in $\ref{hyp:separation}$ of Assumption \ref{hyp-estimation}.

\begin{theorem}
	\label{th:higly_sparse}
	Let $T \in \N$, $s^0  \in \N$ and choose $K\in\N$ such that $s_0
        \leq K$. Let also $\eta \in (0,1)$ and  $r \in \left
          (0,1/\sqrt{2 \, g_\infty \, L_{2}} \right )$. Let
        \begin{equation}
          \label{eq:sep-delta-43}
          \delta^\star \geq \sigma_T\, \Sep
          \quad\text{and}\quad
        \delta^0 \geq \sigma_T\, \Sigma(\eta,r,s^0).
        \end{equation} Assume we observe the random element $y$ of $L^2(\lambda_T)$ under the regression model (\ref{eq:model}) with unknown parameters  $s \in \N$ such that $s \leq K$, $\beta^\star \in (\R^*)^s$ and $\vartheta^\star= \left (
	\theta_1^\star,\cdots,\theta_s^\star\right ) \in \Theta_T^s(\delta^\star)$.  Let $\beta^0 \in (\R^*)^{s^0}$ and $\vartheta^0 =(\theta_1^0,\cdots,\theta_{s^0}^0) \in \Theta_T^{s^0}(\delta^0)$. Suppose that Assumption~\ref{hyp-estimation} holds for the sets $\cq^\star = \{\theta^\star_1,\cdots,\theta_{s}^\star\}  \subset \Theta_T$ of cardinal $s$ and  $\cq^0  = \{\theta^0_1,\cdots,\theta_{s^0}^0\} \subset \Theta_T$  of cardinal $s^0$. Suppose also that the noise process $w_T$ satisfies Assumption~\ref{hyp:bruit} for a noise level $\scale>0$ and a decay rate for the noise variance $\Delta_T>0$.
	
	Then, there exist finite positive constants $\mathcal{C}_0$, 
	$\mathcal{C}_1$, $\mathcal{C}_2$, depending on $r$ and on the function $\Funk$,  such that for the tuning parameter $\kappa$:
	\begin{equation}
	\label{eq:choice_kappa}
	\kappa \geq \mathcal{C}_1 \scale \sqrt{\Delta_T  \log (\tau)}, \quad \text{for some } \tau>1,
	\end{equation}
	the test $\Psi_{\Tau_2}$ using $\Tau_2$ in \eqref{def:test_statistic} satisfies:
	\begin{equation}
	\label{eq:R-sparse-case}
	R_\rho\left(\Psi_{\Tau_2}(\threshold)\right)  \leq 2 \, \mathcal{C}_2 \left (  \frac{|\Theta_T|}{ \sigma_T  \, \tau \sqrt{\log (\tau)} }\vee
	\frac{1}{\tau}\right ), 
	\end{equation}
	for any threshold $\threshold$ and any separation $\rho$ satisfying:
	\begin{equation}
	\label{eq:ineq-rho-t}
	0 < t, \quad \mathcal{C}_0 \sqrt{s^0} \, \kappa \leq \sqrt{\threshold} < \rho \quad \text{  and } \quad  \sqrt{\threshold} + \mathcal{C}_0 \sqrt{s} \, \kappa \leq \rho .
	\end{equation}  
\end{theorem}
The proof can be found in Section~\ref{sec:proof_thm_3_3}. 
\begin{remark}[On the bound $K$]
	The bound $K$ on $s$ is assumed to be known. It is needed to formulate the optimization problem \eqref{eq:generalized_lasso} whose solutions are the estimators of $\beta^\star$ and $\vartheta^\star$. However, we stress that  the constants $\mathcal{C}_0$, 
	$\mathcal{C}_1$, $\mathcal{C}_2$ and the bound on the maximal testing risk do not depend on $K$. Thus, $K$ can be taken arbitrarily large.
\end{remark}

In the next Corollary, we obtain an additionnal upper bound on the minimax separation rate.

\begin{corollary}
	\label{th:higly_sparse_R}
	Under the framework and the assumptions of Theorem \ref{th:higly_sparse} and provided that $|\Theta_T|/\sigma_T \geq 1$, there exist finite positive constants $c$ and $C$, depending on $r$ and the function $\Funk$, such that the minimax separation rate for the test problem \eqref{eq:H0_goodness_of_fit} verifies for any $\alpha \in (0,1)$:
	\begin{equation}
	\label{eq:higly_sparse_R}
	\rho^\star\left ( \alpha \right ) \leq  \rho^{(2)}(\alpha), \quad  \rho^{(2)}(\alpha):= C \,  \scale \, \sqrt{(s\vee s^0\vee 1) \Delta_T \log \left (\frac{c \, |\Theta_T|}{\alpha \, \sigma_T} \right )}.
	\end{equation}
\end{corollary}

\begin{remark}[On the condition $|\Theta_T|/\sigma_T \geq 1$]
	We recall that the set $\Theta_T$ is a compact subset of $\Theta$. In the case where $\Theta$ is the torus $\R/\Z$, $\Theta_T = \Theta$ and the scale parameter $\sigma_T$ tends towards $0$ when $T$ grows to infinity, the condition  $|\Theta_T|/\sigma_T \geq 1$ is satisfied for $T$ large enough. This condition also holds for $T$ large enough in the  Gaussian spikes deconvolution example, with the particular choices for $\Theta_T$ and $\sigma_T$  from  Section \ref{ex:gaussian-spike}, where $\Theta = \R$, $\lim_{T \rightarrow + \infty}\, \Theta_T  = \Theta$ and $\lim_{T \rightarrow + \infty} \, \sigma_T = 0$.
\end{remark}

\begin{proof}[Proof of Corollary \ref{th:higly_sparse_R}]
	Notice that all the assumptions of Theorem \ref{th:higly_sparse} are in force. The result is a direct consequence of Theorem \ref{th:higly_sparse}.
	We fix the tuning parameter $\kappa = \mathcal{C}_1 \scale \sqrt{ \Delta_T \log (\tau)}$ by taking the equality in \eqref{eq:choice_kappa}.
	Then, for
	\begin{equation}
	\label{eq:rho-t:higly_sparse}
	\rho \geq \mathcal{C}_0 \, \sqrt{s\vee 1} \,\kappa  + \sqrt{\threshold} \quad \text{  and } \quad \threshold = \mathcal{C}_0^2 \,  (s^0 \vee 1) \, \kappa^2,
	\end{equation}
	we have \eqref{eq:ineq-rho-t} (in particular $0 < \threshold < \rho$) and by Theorem \ref{th:higly_sparse} for $\tau > 1$:
	\[
	R_\rho^\star \leq R_\rho \left (\Psi_{\Tau_2}(\threshold) \right ) \leq   2 \mathcal{C}_2  \, \left (  \frac{|\Theta_T|}{ \sigma_T  \, \tau \sqrt{\log (\tau)} }\vee
	\frac{1}{\tau}\right ),
	\]
	where the finite positive constants $\mathcal{C}_0$, $\mathcal{C}_1$, $\mathcal{C}_2$, from  Theorem \ref{th:higly_sparse} depend on $r$ and $\Funk$.
	
	Then, taking $\tau =  c |\Theta_T|/(\alpha \sigma_T)$ with $c =  (2\mathcal{C}_2) \vee \rm e $ and using that by assumption $|\Theta_T|/\sigma_T \geq 1$,  we get for $\rho \geq \sqrt{2}\mathcal{C}_0 \mathcal{C}_1 \scale \sqrt{(s + s^0) \vee 2} \sqrt{ \Delta_T \log (c |\Theta_T|/(\alpha \sigma_T))}$ and $\alpha \in (0, 1)$ that $R_\rho^\star \leq  \alpha$.
	We readily deduce \eqref{eq:higly_sparse_R} with $C = 2\mathcal{C}_0 \mathcal{C}_1$.	
\end{proof}

\begin{remark}[Combining the upper bounds of Corollaries \ref{th:non_sparse_R} and \ref{th:higly_sparse_R}]
	Let $\alpha \in (0,1)$. Suppose that the assumptions of Corollaries \ref{th:non_sparse_R} and \ref{th:higly_sparse_R} hold.  
	Previous results show that each procedure may perform better than the other one in convenient regimes of the parameters, involving the unknown parameter $s$. In order to aggregate the two procedures into an automatic one, we take the maximum of the two test procedures. This aggregated test procedure rejects as soon as at least one of the procedures rejects, and accepts otherwise. 
	
	More precisely, let $\rho^{(1)}(\alpha/2)$ be defined by \eqref{eq:non_sparse_R} with $\alpha$ replaced by $\alpha/2$ and  set $t^{(1)} = (\rho^{(1)}(\alpha/2))^2/2$; and let  $\rho^{(2)}(\alpha/2)$ be defined in \eqref{eq:higly_sparse_R} and $t^{(2)}$ be given by \eqref{eq:rho-t:higly_sparse} with $\alpha$ replaced by $\alpha/2$. Then,  Corollaries \ref{th:non_sparse_R} and \ref{th:higly_sparse_R} imply that $R_{\rho^{(1)}} \left (\Psi_{\Tau_1}(\threshold^{(1)}) \right ) \leq \alpha/2$ and  $R_{\rho^{(2)}} \left (\Psi_{\Tau_2}(\threshold^{(2)}) \right ) \leq \alpha/2$. We define the test: 
	\[
	\Psi^{\max} = \max(\Psi_{\Tau_1}(\threshold^{(1)}),\Psi_{\Tau_2}(\threshold^{(2)})).
	\]	
	It is straightforward to see that the type I error probability satisfies:
	\[
	\sup_{(\beta^\star, \vartheta^\star) \in H_0}\E_{(\beta^\star,\vartheta^\star)}[\Psi^{\max}] \leq \alpha.
	\]
	Moreover, we have for $\rho^{\min}(\alpha) = \rho^{(1)}(\alpha/2) \wedge \rho^{(2)}(\alpha/2)$ the following bound on the type II error probability:
	\begin{equation*}
	\sup_{(\beta^\star, \vartheta^\star) \in H_1(\rho^{\min})} \E_{(\beta^\star,\vartheta^\star)}[1-\Psi^{\max}] \leq \alpha/2.
	\end{equation*}
	Therefore, we deduce an upper bound on $\rho^\star(\alpha)$ of order $\rho^{\min}(\alpha)$, that is:
	\begin{equation}
	\label{eq:rates_detection}
	\rho^{\min}(\alpha) = \min\left (\left ( \frac{80 \varmodel}{\alpha} \right )^{1/4} , C\scale \,\sqrt{(s \vee s^0 \vee 1) \Delta_T \log \left (\frac{2 \, c \, |\Theta_T|}{\alpha \, \sigma_T} \right )} \right ),
	\end{equation}
	for a positive constant $c \geq 2$.
	We identify two regimes depending on whether the number of features of the observed signal 
        is sufficiently small or not. 
	Indeed, we notice that when $\alpha$ is fixed and:
	$$ 
	s\vee s^0 \vee 1 \ll \left(\frac{\varmodel}{\alpha} \right )^{1/2} \cdot \left (\scale^2\Delta_T \log \left (\frac{ 2 \, c \, |\Theta_T|}{\alpha \, \sigma_T} \right ) \right )^{-1},
	$$ Corollary \ref{th:higly_sparse_R} yields a sharper upper bound on the separation rate than Corollary \ref{th:non_sparse_R}.
\end{remark}

\subsection{Minimax separation rates for signal detection}
\label{sec:signal-detection}
We   illustrate   our  results   on   a   simple  model   motivated   by
\cite{Ingster2010}  for   sparse  linear   regression.  We   consider  a
discrete-time  process $y$  over  a regular  grid  $t_1<\cdots< t_T$  on
$\Theta= \R  / \Z$ with grid  step $\Delta_T = 1/T$.  We set $\lambda_T$
and         $w_T$          as         in~\eqref{eq:def-lambdaT-reg-grid}
and~\eqref{eq;def-wT-reg-grid} from Section \ref{sec:discrete_noise}. We
recall that  $\varmodel = 2  \scale^4 \Delta_T^2 T$ where  $\scale>0$ is
the  noise level.  In  the  following, we  assume  without  any loss  of
generality that $\scale=1$.  \medskip

Let us consider the framework of signal detection when $s^0=0$.
Under the assumptions of Corollary \ref{th:higly_sparse_R}, the test problem \eqref{eq:H0_goodness_of_fit} reduces to: 
\begin{equation}
\label{eq:signal_detection}
\begin{cases}
H_0 : &\beta^\star =0,\\
H_1 (\rho) : &(\beta^\star, \vartheta^\star) \in (\R^*)^s \times
\Theta_T^{s}(\delta^\star) \quad \text{s.t.}\quad  \norm{\beta^\star \Phi_{T}(\vartheta^\star)}_{L^2(\lambda_T)}\geq \rho.
\end{cases}
\end{equation}
Moreover,  under the  assumptions  of Corollary  \ref{th:higly_sparse_R}
(which   in  particular   gives   a  lower   bound  on   $\delta^\star$,
see~\eqref{eq:sep-delta-43})  and  with  the   same  arguments  used  to
establish      \eqref{eq:small-eig}      in       the      proof      of
Lemma~\ref{lem:identifiability}, we can show that:
\begin{equation}
5/6 \leq C_{\min} := \min_\beta \, \frac{\norm{\beta\Phi_{T}(\vartheta^\star)}_{L^2(\lambda_T)}}{\norm{\beta}_{\ell_2}}, \quad \! \! C_{\max} := \max_\beta \, \frac{\norm{\beta\Phi_{T}(\vartheta^\star)}_{L^2(\lambda_T)}}{\norm{\beta}_{\ell_2}} \leq 7/6. 
\end{equation}
Therefore, the separation in the alternative hypothesis $H_1(\rho)$ can be formulated as a lower bound on $\norm{\beta^\star}_{\ell_2}$ since we have: $$ C_{\min }\norm{\beta^\star}_{\ell_2} \leq \norm{\beta^\star \Phi_{T}(\vartheta^\star)}_{L^2(\lambda_T)} \leq C_{\max}\norm{\beta^\star}_{\ell_2}.$$

We set $\Theta_T = \Theta$ and thus $|\Theta_T| = 1$. We get from \eqref{eq:rates_detection} the following upper bound on $\rho^\star(\alpha)$ for any $\alpha \in (0,1)$:
\begin{equation}
\rho(\alpha) = C \min\left ( \frac{1}{(\alpha T)^{\frac{1}{4}}} , \sqrt{\frac{s}{T}\log \left (\frac{c}{\alpha \, \sigma_T} \right )} \right ),
\end{equation}
with $C$ a finite positive constant.
Let $(\alpha_{T}, T \geq 1)$ be a $(0, 1)$-valued sequence  which converges  to zero when $T$  grows to infinity. We deduce that:
\[
\lim_{s,T \to +\infty} R_{ \rho(\alpha_{T})}^\star = 0.
\]
By letting the sequence $(\alpha_{T}, T \geq 1)$ converge towards $0$ as slow as we want, we deduce that for a sequence of separations $(\rho_{s,T}, \, T \geq 1, s \geq 1)$ such that:
\begin{equation}
\label{eq:rate-continuous-dico}
\lim_{s,T \to +\infty} \frac{\rho_{s,T}}{  \frac{1}{ T^{\frac{1}{4}}} \wedge \sqrt{\frac{s}{T}\log \left (\frac{c}{\sigma_T} \right )} }= + \infty ,
\end{equation}
we have { $\rho_{s,T}\geq  \rho(\alpha_T)$ and thus:}
\begin{equation*}
\lim_{s,T \to +\infty} R_{\rho_{s,T}}^\star = 0.
\end{equation*}
Hence, we have obtained an asymptotic upper bound of the minimax separation associated to the detection of a finite linear combination of features issued from a continuous dictionary. 
\medskip

We now compare this upper bound  to the asymptotic lower bound obtained in the case where  the dictionary contains a finite number of features instead of a continuum. 
Assume that the dictionary is fixed, known and contains $p$ features parametrized by the parameters in the known and fixed set $ \cq^0=\{\theta_1^0,\cdots, \theta_p^0\} \subset \Theta_T$.
We consider the high dimensional linear regression model:
\begin{equation*}
y = \beta^\star \Phi_T(\vartheta^0) + w_T \quad  \text{in $L^2(\lambda_T)$},
\end{equation*}
with $\vartheta^0 = (\theta_1^0,\cdots,\theta_p^0) \in \Theta_T^p$ and where $\beta^\star \in \R^p$ is a $s$-sparse vector. Notice that in this model the entries of $\beta^\star$ can take the value 0.  
The high dimension comes from the fact that $p$ can be much larger than $T$. Under coherence assumptions on the finite dictionary and 
for a sequence of separations $(\rho_{s,T}, \, T \geq 1, s \geq 1)$ such that:
\begin{equation}
\label{eq:rate-finite-dico}
\lim_{s,T \to +\infty}  \frac{\rho_{s,T}}{\frac{1}{T^{\frac{1}{4}}} \wedge \sqrt{\frac{s}{T}\log (p)} \wedge \frac{p^{\frac{1}{4}}}{\sqrt{T}} }=0,
\end{equation}
the authors of \cite{Ingster2010} showed for different hypotheses on the design matrix $\Phi_T(\vartheta^0)$ that:
$$
\lim_{s,T \to +\infty} R_{\rho_{s,T}}^\star = 1.
$$
It means that the hypotheses \eqref{eq:signal_detection} cannot be distinguished asymptotically when the separation converges to zero faster than the rate given by \eqref{eq:rate-finite-dico}.

\begin{remark}[Comparison between the rates obtained for finite and continuous dictionaries]
In the high-dimensional linear case (i.e., \( T \leq p \)), given that \( 1 /T^{1/4} \leq p^{1/4}  / \sqrt{T} \), the asymptotic minimal intensity allowing signal detection given by \eqref{eq:rate-finite-dico} becomes:
$$
\frac{1}{T^{\frac{1}{4}}} \wedge \sqrt{\frac{s}{T}\log (p)}.
$$
This rate matches, up to a logarithmic factor, the rate given by
\eqref{eq:rate-continuous-dico} for our more general model. There are two distinct regimes: the sparse case (\( s \leq \sqrt{T} / \log(p) \)) and the non-sparse case. Additionally, the magnitude of the size \( p \) of the finite dictionary plays an analogous role as the quantity \( 1/\sigma_T \) that appears in the logarithmic terms. The term \( 1/\sigma_T \) is of the order of the maximal number of shifted elements permissible in our mixture, considering a separation condition of order $\sigma_T$ and shift parameters within a compact set possibly growing with $T$.
\end{remark}

%%%%%%%%%%%%%%%%%%%%%%%%
\section{Goodness-of-fit of the dictionary}
\label{sec:goodness-of-fit-features} 

In spectroscopy, a prescribed material has known chemical components and a list of $s_0$ corresponding location parameters of the features is provided. From a sampled material we want to decide whether its chemical components are included in the prescribed list. The linear coefficients are non-negative in this case and they are not given, which makes the null hypothesis composite, that is, fixed location parameters and varying positive linear coefficients.
We generalize this setup to real valued linear coefficients. Under the null hypothesis the location parameters are still fixed, but the linear coefficients vary with fixed sign.

More precisely, let $s^0 \in \N$ and let $\cq^0 = \{\theta_1^0, \cdots, \theta_{s_0}^0\} \subset \Theta_T$ be a set of known location parameters pairwise separated by a distance $\delta^0 \geq 0$ so that the model is identifiable, see Lemma \ref{lem:identifiability}.  We set the vector $\vartheta^0 = (\theta_1^{0},\cdots, \theta_{s^0}^0)$.
We include in the null hypothesis all linear combinations:
$$
\sum_{j=1}^{s^0} \beta_j^\star \varphi_T(\theta_j^0)
$$
with $ \beta_j^\star $ being either 0 or with the same sign as $\beta_j^0$, for all $j$ from 1 to $s^0$. Thus we split the set $\cq^0$ into $\cq^{0,+}$ and $\cq^{0,-}$, those parameters $\theta_k^0$ associated to $\beta_k^0 >0$ and to $\beta_k^0<0$, respectively: 
$$
\cq^{0, \epsilon}=\{\theta^0_k\in \cq^0\, \colon\, \epsilon 
\beta^0_k > 0\}, \quad \epsilon  \in \{+,-\}.
$$
%Let $v^0=(v^0_1, \ldots, v^0_{s^0})$ be a vector in $\{-1,1\}^{s^0}$ that contains the signs of all linear coefficients under the null hypothesis: $v_j^0 = \text{sign}(\beta_j^0)$. Consider two disjoint subsets of the set $\cq^0$  associated to linear coefficients with sign $\epsilon=\pm 1$: $\cq^{0, \epsilon}= \{\theta^0_k\in \cq^0\, \colon\, \epsilon v^0_k > 0\}$. 
Let $s \in \N^*$. Assume that we observe a random element $y$ issued from the model \eqref{eq:model} with linear coefficients $\beta^\star\in (\R^*)^s$ and non-linear parameters $\vartheta^\star = ( \theta_1^\star, \cdots,\theta_s^\star) \in \Theta_T^s$.
We  test whether the unknown set:
$$
\cq^{\star, \epsilon}= \{\theta^\star_k\in \cq^\star\, \colon\,
\epsilon  \beta^\star_k>0 \} \text{ is a subset of }\cq^{0,
  \epsilon}\text{ for each } \epsilon\in \{+, -\}.
$$
If $s^0=0$, this amounts to testing that $\cq^\star$ is empty, which corresponds to the signal detection framework presented in Section \ref{sec:goodness-of-fit} in the case $s^0=0$. Hence, we shall assume in this section that $s_0 \geq 1$. 

For example, in spectroscopy, $\cq^{0, -}$ is empty because all linear parameters are positive and this amounts to testing that the present chemical elements are in the prescribed list $\cq^{0}$ but they may appear with various positive linear coefficients (amplitudes). Under the alternative, other chemical components are present (located at unknown frequencies not in the prescribed list).

\subsection{A measure of discrepancy between dictionaries}
We define  the closed balls centered at $\theta\in
\Theta_T$ with radius $r$ by:
\[
\mathcal{B}_{T}(\theta,r) = \left\{ \theta' \in \Theta_T\,\colon\, 
\dT(\theta, \theta') \leq r\right\}  \subseteq \Theta_T.
\]
Let us define for $\epsilon\in \{+, -\}$ the set of indices
$\mathcal{I}^\epsilon = \{k\in \{1, \ldots, s^0\}, \, \epsilon \beta^0_k > 0\}$.
We introduce for $r> 0$,  $\epsilon\in \{+, -\}$ and $k \in
\mathcal{I}^\epsilon$,  the set $S_k^\epsilon(r)$ gathering the indices of the elements of $\cq^{\star, \epsilon}$ that are close to the element $\theta_k^0$ of $\cq^{0, \epsilon}$:
\begin{equation}
\label{def:sets_S0}
S_k^\epsilon(r) = \left \{\ell\in \{1,\cdots,s\}: \theta^\star_\ell \in
  \mathcal{B}_{T}(\theta_k^0,r) \text{ and } \sgn (\beta^\star_\ell)=1 \right \}.
\end{equation}
Notice that the sets $S_k^\epsilon(r)$ can be empty. Furthermore, we assume that $ r < \min_{\ell\neq k} \dT(\theta^0_\ell,\theta^0_k)/2$  so that
the sets $S_k^\epsilon(r)$ with $\epsilon\in \{+, -\}$ and $k \in \mathcal{I}^\epsilon$ are pairwise disjoint.
We also set:
\[
S(r) = \bigcup_{\epsilon\in \{+, -\}} S^{\epsilon}(r)\quad \text{with} \quad S^\epsilon(r) =  \bigcup_{k \in \mathcal{I}^\epsilon} S_{k}^\epsilon(r).
\]

We now define a discrepancy measure between the model and any approximation by a linear combination of features having their non-linear parameters in $\cq^0 $ and the linear parameters with the same signs, for $r >0$:
\begin{equation*}
\mathcal{D}_{T,r}(\beta^\star,\vartheta^\star,v^0, \vartheta^0)  =
\sum_{\epsilon\in \{+, -\}}\, \sum_{k \in \mathcal{I}^\epsilon} \,\sum\limits_{\ell\in
	S_k^ \epsilon(r)}  
| \beta_{\ell}^\star| \, \dT(\theta_{\ell}^\star,\theta_{k}^{0})^{2} + \sum\limits_{k \in
	S(r)^{c}}|\beta_{k}^\star|,
\end{equation*}
where $S(r)^{c}$ denotes the complementary set of $S(r)$ in $\{1,
\ldots, s\}$ and $v^0=(v^0_1, \ldots, v^0_{s^0})$ contains the signs of
all linear coefficients $\beta^0$, $v_j^0 = \sgn  (\beta_j^0)$. 
Notice that $\mathcal{D}_{T,r}(\beta^\star,\vartheta^\star,v^0, \vartheta^0) = 0$ if and only if $\cq^{\star,+} \subseteq \cq^{0, +}$ and  $\cq^{\star,-} \subseteq \cq^{0, -}$. 

\subsection{The testing hypotheses}

We shall test the following hypotheses:
\begin{equation}
\label{eq:H0-dico}
\begin{cases}
H_0 : &(\beta^\star, \vartheta^\star) \in (\R^*)^s \times \Theta_T^{s}(\delta^\star), \quad  \quad  \cq^{\star,+} \subseteq \cq^{0, +} \text{ and  } \cq^{\star,-} \subseteq \cq^{0, -},\\ 
H_1(\rho) :  &(\beta^\star, \vartheta^\star) \in (\R^*)^s\times \Theta_T^{s}(\delta^\star) \quad \text{  and } \quad \mathcal{D}_{T,r}(\beta^\star,\vartheta^\star,v^0, \vartheta^0) \geq\rho,
\end{cases}
\end{equation}
where $\rho$ and $\delta^\star$ are separation parameters depending \emph{a priori} on $T$, $s$ and $s^0$ that need to be evaluated.
Notice that the null hypothesis is also composite.
We recall the definitions \eqref{def:minimax_testing_risk} and \eqref{def:minimax_separation} of the minimax testing risk $R_\rho^\star$ and the minimax separation $\rho^\star$. In the following, we give upper bounds on the testing risk and on the minimax separation $\rho^\star(\alpha)$ for any $\alpha \in (0,1)$.

\subsection{Main result}
In this section, we build a test for  \eqref{eq:H0-dico}.
Under Assumptions \ref{hyp:reg-f} and \ref{hyp:g>0}, we define the element of $L^2(\lambda_T)$:
\begin{equation}
\label{eq:def_p0}
p_{0} = \sum\limits_{k=1}^{s^0} \alpha_{k}
\phi_{T}(\theta^{0}_{k})
+ \sum\limits_{k=1}^{s^0} \coeff_{k} \, 
\tilde D_{1,T}[\phi_{T}](\theta^{0}_{k}) ,
\end{equation}
where $\alpha, \coeff \in \R^{s^0}$ solve the system:
\begin{equation}
\label{matrix_system1}
\left \langle \phi_T(\theta^0_k),p_0 \right \rangle_{L^2(\lambda_T)}= \sgn (\beta^0_k) \text{  and }   \left \langle \partial_\theta \phi_T(\theta^0_k),p_0 \right \rangle_{L^2(\lambda_T)} = 0 , \quad \forall k \in \{1,\cdots,s^0\} .
\end{equation}

\begin{remark}
	The element $p_0$ of $L^2(\lambda_T)$ coincides with the vanishing derivative pre-certificate which appears in \cite[Section 4]{duval2015exact} and is the solution of \eqref{matrix_system1} with minimal norm $\norm{p_0}_{L^2(\lambda_T)}$. We state in Lemma~\ref{lem:interpolating_certificate} the existence of such function and prove its further properties used in the following result.
\end{remark}

Using the estimator $\hat{\beta}$ from \eqref{eq:generalized_lasso} for a given value of the tuning parameter $\kappa$, we  define the test statistic:
\begin{equation}
\label{def:test_statistic_dico}
\Tau_3 = \norm{\hat{\beta}}_{\ell_1} - \left \langle y,p_0 \right \rangle_{L^2(\lambda_T)}.
\end{equation}
and the corresponding test $\Psi_{\Tau_3}(\threshold) = \ind_{\{ |\Tau_3| > \threshold\}}$. 
Thus we use the certificate function as a filter of the signal and note that $\E\left \langle y,p_0 \right \rangle_{L^2(\lambda_T)} = \norm{\beta^\star}_{\ell_1}$ under the null hypothesis.

\begin{theorem}
	\label{theorem:inclusion}
	Let $T \in \N$, $s^0  \in \N^*$ and choose $K\in\N$ such that
        $s_0 \leq K$. Let also $\eta \in (0,1)$ and  $r \in \left
          (0,1/\sqrt{2 \, g_\infty \, L_{2}} \right )$. Let:
        \[
          \delta^\star \geq \sigma_T\, \Sep \quad\text{and}\quad
          \delta^0 \geq \sigma_T\, \Sigma(\eta,r,s^0).
        \]
        Assume  we observe  the random  element $y$  of $L^2(\lambda_T)$
        under  the   regression  model  (\ref{eq:model})   with  unknown
        parameters   $s    \in   \N^*$    such   that   $s    \leq   K$,
        $\beta^\star            \in             (\R^*)^s$            and
        $\vartheta^\star=                     \left                    (
          \theta_1^\star,\cdots,\theta_s^\star\right        )        \in
        \Theta_T^s(\delta^\star)$.  Let $v^0  \in  \{-1,1\}^{s^0}$ be  a
        sign                vector                and                let
        $\vartheta^0=(\theta^0_1,\cdots,\theta_{s^0}^0)\in
        \Theta_T^{s^0}(\delta^0)$.      Suppose     that      Assumption
        \ref{hyp-estimation}       holds        for       the       sets
        $\cq^\star =  \{\theta^\star_1,\cdots,\theta_{s}^\star\} \subset
        \Theta_T$          of          cardinal         $s$          and
        $\cq^0 = \{ \theta^0_1,\cdots,\theta_{s^0}^0\} \subset \Theta_T$
        of cardinal  $s^0$.  Suppose also  that the noise  process $w_T$
        satisfies   Assumption  \ref{hyp:bruit}   for   a  noise   level
        $\scale>0$ and a decay rate for the noise variance $\Delta_T>0$.
	
	Then, the test statistic $\Tau_3$ is uniquely defined and there exist finite positive constants, 
	$a$ and $\mathcal{C}_i$ with $i=1,\cdots,5$,
	(depending on $r$ and on the function $\Funk$) such that for 
	any $\tau > 1$ and any tuning parameter $\kappa$:
	\begin{equation}
	\label{eq:choice_kappa_dico}
	\kappa \geq \mathcal{C}_1 \scale \sqrt{\Delta_T  \log (\tau)},
	\end{equation} 
	the test $\Psi_{\Tau_3}$ satisfies:
	\begin{equation}
	\label{eq:errorI+II_inclusion}
	R_\rho \left (\Psi_{\Tau_3}(\threshold) \right )  \leq 2 \, \mathcal{C}_2 \left (  \frac{|\Theta_T|}{ \sigma_T \,\tau \, \sqrt{\log (\tau)} }\vee
	\frac{1}{\tau}\right ) +   \frac{2}{\tau^{a \, s_0}},
	\end{equation}
	for any threshold $\threshold > 0$ and any separation $\rho > 0$ satisfying:
	\begin{equation}
	\threshold \geq  2\,\mathcal{C}_3\, s^0 \, \kappa  \quad \text{ and } \quad \rho \geq \mathcal{C}_4 \,s \, \kappa + \mathcal{C}_5 \, \threshold.
	\end{equation}  		
\end{theorem}
The proof is given in Section~\ref{sec:proof_thm_4_2}.

\subsection{Separation rates}

We give in this section an upper bound on the minimax separation $\rho^\star$ to test the goodness-of-fit of the dictionary, that is to distinguish the assumptions $H_0$ and $H_1(\rho)$ presented in Section~\ref{sec:goodness-of-fit-features}.
\begin{corollary}
	\label{theorem:inclusion_R}
	Under the framework and the assumptions of Theorem \ref{theorem:inclusion}, there exist finite positive constants $c$ and $C$ (depending on $r$ and the function $\Funk$) such that provided that $|\Theta_T|/\sigma_T \geq 1$, we have for any $\alpha \in (0,1)$:
	\begin{equation}
	\label{eq:inclusion_R}
	\rho^\star(\alpha) \leq C \,  \scale \, (s\vee s^0) \, \sqrt{ \Delta_T \log \left (\frac{c \, |\Theta_T|}{\alpha \, \sigma_T} \right )}.
	\end{equation}
\end{corollary}

\begin{proof}
	The result is a direct consequence of Theorem \ref{theorem:inclusion}. 
	We fix the tuning parameter $\kappa = \mathcal{C}_1 \scale \sqrt{\Delta_T \log(\tau)}$ by taking the equality  in \eqref{eq:choice_kappa_dico}. 
	Then, for $\rho \geq \mathcal{C}_4 \, s \, \kappa +\mathcal{C}_5 \, \threshold$ and $\threshold = 2 \, \mathcal{C}_3 \, s^0 \kappa $ we have by Theorem \ref{theorem:inclusion} for $\tau > 1$ and since $s_0 \geq 1$:
	\[
	R_\rho^\star \leq R_\rho \left (\Psi_{\Tau_3}(\threshold) \right ) \leq 2\mathcal{C}_2 \left (  \frac{|\Theta_T|}{ \sigma_T \,\tau \, \sqrt{\log (\tau)} }\vee
	\frac{1}{\tau}\right ) +   \frac{2}{\tau^a},
	\]
	where the finite positive constants 
	$a$, $\mathcal{C}_i$ with $i \in \{1,\cdots,5\}$, from  Theorem \ref{theorem:inclusion} depend on $r$ and the function $\Funk$.
	
	Hence,     by      taking     $\tau     =      c'     /(\sigma_T
        \alpha/(2|\Theta_T|))^{c''}  $ with  $c''  = 1  \vee (1/a)$  and
        $c' =  (2\mathcal{C}_2) \vee \rm  e \vee 2^{1/\textit{a}}  $, we
        get                                                          for
        $\rho \geq  2\mathcal{C}_1((2\, \mathcal{C}_3  \, \mathcal{C}_5)
        \vee \mathcal{C}_4)\scale(s\vee  s^0) \sqrt{ \Delta_T \log  ( c'
          /(\sigma_T      \alpha/(2|\Theta_T|))^{c''}       )}$      and
        $\alpha           \in           (0,           1)$           that
        $R_\rho^\star   \leq  \frac{\alpha}{2}   +  \frac{\alpha}{2}   =
        \alpha$.    We   then    deduce   \eqref{eq:inclusion_R}    with
        $c= 2 c^{\prime (1/c'')}$.
\end{proof}

\section{Gaussian scaled-spikes deconvolution}
\label{ex:gaussian-spike}

In this section, we consider the discrete time process observed on a
regular grid  of $\R$ given in Section \ref{sec:discrete_noise}.  We recall that
Assumption \ref{hyp:bruit} holds with:
\[
\lambda_T =
\Delta_T\sum_{j=1}^T \delta_{t_j}
\quad\text{with}\quad
t_j=-a_T + j \Delta_T
\quad\text{and}\quad
\Delta_T=\frac{2a_T}{T},
\]
and
$w_T$ given by~\eqref{eq;def-wT-reg-grid}, where   $T \in \N^*$.
We consider the scaled Gaussian features  associated to the function:
\[
\kernel(t,\sigma) \mapsto  \frac{ \exp(-t^2/2\sigma^2) }{\pi^{1/4} \sigma ^{1/2} }
\quad\text{defined on}\quad
\Theta \times \mathfrak{S} = \R \times \R_+^*.
\]
We shall see below that the natural choice for the function $\Funk$
appearing in~\eqref{eq:def_K-app} is given by:
\[
\Funk = \kernel^0 \ast \kernel^0 = \pi^{1/4} \kernel^0(\cdot / \sqrt{2})
\quad\text{with}\quad  \kernel^0(\cdot)=\kernel(\cdot, 1). 
\]

In  the   following,  we  check  that   Assumption  \ref{hyp-estimation}
holds. Then, using Theorem \ref{maintheorem}  on a particular example, we
provide     a    prediction     bound    for     the    estimator     of
$(\beta^\star,  \vartheta^\star)$ solution  of the  optimization problem
\eqref{eq:generalized_lasso}.

\subsection{Choice of the approximating kernel}

We  denote  the   unscaled  feature $\varphi^0$  on   $\theta  \in  \Theta$
by:
\[
\varphi^0(\theta)  = \kernel(\theta -  \cdot, 1)= \kernel^0(\theta -\cdot).
\]
We  define the
mapping      $f_T     :      \Theta      \rightarrow     \Theta$      by
$f_T(\theta)= \theta  / \sigma_T$  for any $\theta  \in \Theta$  and the
(pushforward) measure  $\lambda_T^0 = \lambda_T \circ  f_T^{-1}$ so that
for any $g \in L^1(\lambda_T^0)$:
\[
\int g(\theta/\sigma_T) \lambda_T(\rd \theta) = \int g(\theta)  \,  \lambda^0_T(\rd \theta).
\]
The Hilbert space $L^2(\lambda_T^0)$ is  endowed with its natural scalar
product $\left  \langle \cdot, \cdot  \right \rangle_{L^2(\lambda_T^0)}$
and norm $\norm{\cdot}_{L^2(\lambda_T^0)}$.  We define on $\Theta^2$ the
kernel:
\begin{equation*}
\cK_T^0(\theta, \theta')= \langle \phi_{T}^0(\theta), \phi_{T}^0(\theta')
\rangle_{L^2(\lambda_T^0)}
\quad \text{  with  } \quad
\phi_{T}^0(\theta) = \varphi^0(\theta) / \norm{\varphi^0(\theta)}_{L^2(\lambda_T^0)}.
\end{equation*}
The kernel $\ck_T$ can be seen as a scaled kernel derived from $\ck_T^0$ as for $\theta, \theta' \in \Theta$:
\[
\cK_T(\theta, \theta') = \cK_T^0(\theta/\sigma_T, \theta'/\sigma_T).
\]
When the measure $\lambda_T^0$ converges in some sense, as $T$ goes to
infinity, towards the Lebesgue measure $\Leb$ on $\R$, it is natural to
consider the approximation $\ck_\infty^0$ of $\ck_T^0$ on $\Theta^2$ by:
\begin{equation*}
\ck_\infty^0(\theta,\theta') =\left \langle\phi_\infty^0(\theta) , \phi_\infty^0(\theta')
\right \rangle_{L^2(\Leb)}
\quad \text{with} \quad
\phi_\infty^0(\theta) = \varphi^0(\theta) / \norm{\varphi^0(\theta)}_{L^2(\Leb)}.
\end{equation*}
Thanks to the definition of $\Funk$, we also
have on $\Theta^2$ that:
\begin{equation*}
\Funk (\theta - \theta') = \ck_\infty^0(\theta,\theta').
\end{equation*}
The approximating kernel $\ck_T^{\text{prox}}$ is then  given
by \eqref{eq:def_K-app} on $\Theta^2$, that is, 
$\ck_T^{\text{prox}}(\cdot, \cdot)=
\ck_\infty^0(\cdot/\sigma_T,\cdot/\sigma_T) $. 

\subsection{Checking Assumption~\ref{hyp-estimation}}

\subsubsection{Regularity           of            the           dictionary}

We   refer  to   \cite[Section   8]{butucea22}  to   check  that
Assumption  \ref{hyp-estimation}~$\ref{hyp:theorem_properties_dico}$ 
holds for the feature $\varphi_T$ defined by~\eqref{eq:def-features} and
any scale parameter $ \sigma_T\in \mathfrak{S}=\R_+^*$.

\subsubsection{Boundedness and local concavity on the diagonal}
Elementary  calculations show  that  $g_\infty =  -  \Funk''(0) =  1/2$.
By definition of $\Funk$, we directly deduce that Assumption~\ref{hyp:properties_k} holds. 
We also get that for $r\in (0, \sqrt{2})$:
\[
\varepsilon (r) = 1 - \expp{-r^2/4} > 0
\quad\text{and} \quad
\nu(r) = \left ( 1 - \frac{r^2}{2} \right) 
\expp{-r^2/4}.
\]
We fix $r \in (0,1/2)$. We readily check that
Assumption~\ref{hyp-estimation}~$\ref{hyp:theorem_properties_k}$  is 
verified.

\subsubsection{Proximity to the approximating kernel}
In order for the  kernel $\ck_T^{\text{prox}}$ to be a good
approximation of $\ck_T$ in the sense of
Assumption~\ref{hyp-estimation}~$\ref{hyp:proximity-setting}$, we shall
consider  the set $\Theta_T$ over which the optimization is performed:
\begin{equation*}
\Theta_T = [-(1-\xi) a_T, (1-\xi)  a_T] \subset [-a_T,a_T], 
\end{equation*}
with a given shrinkage parameter $\xi \in (0,1)$. 
Intuitively,  one  does  not  expect  the  estimation  of  the  location
parameter  to perform  well  near  the lower  and  upper  bounds of  the
observation  grid (given  by  the support  of $\lambda_T$).   Following
\cite[Section~8]{butucea22}, we set:
\begin{equation}
\label{eq:shrink}
\gamma_T=2 \Delta_T \, \sigma_T^{-1}+ \sqrt{\pi}\,\expp{-\xi^2
	a_T^2/2\sigma_T^2}.
\end{equation}

Recall $\DT $ and $\CT$ defined by~\eqref{eq:def-rho} and
\eqref{def:V_1}. 
Using Lemma \cite[Lemma 8.1]{butucea22}, there exist finite positive
universal constants $c_0$, $c_1$ and $c_2$, such that $\gamma_T<c_0$
implies: 
\begin{equation}
\DT \leq c_1 \gamma_T
\quad \text{and} \quad
|1 - \CT|\leq  c_2 \gamma_T.
\end{equation}
Assume that $(a_T, T\geq 2)$ and $(\sigma_T, T \geq 2)$ are sequences of
positive numbers, such that:
\begin{equation}
\label{eq:lim-D-b_T}
\lim_{T\rightarrow \infty } a_T   = + \infty, \quad \lim_{T\rightarrow \infty } \sigma_T   = 0
\quad\text{and}\quad
\lim_{T\rightarrow \infty } \Delta_T \, \sigma_T^{-1} = 0. 
\end{equation}
Therefore, we have $\lim_{T \rightarrow +\infty} \DT = 0$ and $\lim_{T
	\rightarrow +\infty} \CT = 1$.
\medskip

Let    $\eta     \in    (0,1)$    be    fixed.      We    deduce    that
under~\eqref{eq:lim-D-b_T},                                   Assumption
\ref{hyp-estimation}~$\ref{hyp:proximity-setting}$ is  satified provided
that $T$  is larger  than some  constant depending  on $\eta$,  $r$, the
sparsity    $s$   and    the    sequences   $(a_T,    T\geq   2)$    and
$(\sigma_T, T \geq 2)$.

\subsubsection{Separation of the non-linear parameters}

We                              remark                              that
$\lim_{r''\rightarrow \infty } \sup_{|r'|\geq r''} |\Funk^{(i)}(r')|=0$ for
all    $i\in   \{0,\ldots,    3\}$.   Thus,    we   deduce    from   the
definition~\eqref{eq:def-delta-rs}  of  $\delta$ that  $\delta(u,s)$  is
finite   for  all   $s\in  \N^*$   and  $u>0$.    Let  us   stress  that
$\sup_{s\in  \N^*}  \delta(u,s)  \leq  M/u$ for  some  universal  finite
constant $M$, see  \cite[Remark~8.2]{butucea22}. Therefore, the quantity
$\Sep$ is bounded by a constant depending only on $\eta$ and $r$.
\medskip

So Assumption \ref{hyp-estimation}~$\ref{hyp:separation}$ is verified as
soon as  $|\theta - \theta' |> \sigma_T \, \Sep$ for all for all $\theta
\neq \theta' \in \cq^\star$. (Notice this happens for the scaling
parameter $\sigma_T$ small  enough
depending on $\cq^\star$.)

\subsection{Prediction error bound in a  particular case}
Recall the shrinkage parameter $\xi\in (0, 1)$ in~\eqref{eq:shrink}. Let us assume that:
\[
a_T = \log(T)
\quad\text{and}\quad
\sigma_T = 1/ \sqrt{\xi \log(T)}.
\]
In particular, condition~\eqref{eq:lim-D-b_T} holds.  In this case, there
exists a finite positive constant $c$ depending on $r$, $\eta$ and $\xi$
such   that   for  $T   \geq   c   \log(T)^{3/2}  \,   s$,   Assumption
\ref{hyp-estimation}  holds   (notice  that  the   separation  condition
\eqref{eq:separation-Q-star} of  the location parameters  in $\cq^\star$
is  also verified  for $T$  large enough,  depending on  $\cq^\star$, as
$\lim_{T\rightarrow  + \infty } \sigma_T=0$).  By Theorem \ref{maintheorem}
with   $\tau   =   T$   and   $\kappa$  given   by   the   equality   in
\eqref{eq:bound-kappa}, we get that:
\[
\frac{1}{\sqrt{T}}\norm{\hat{\beta}\Phi_{T}(\hat{\vartheta}) -
	\beta^{\star}\Phi_{T}(\vartheta^{\star}) }_{\ell_2} \leq
\mathcal{C}_0 \,\mathcal{C}_1 \, \scale \, \sqrt{\frac{s\log (T)}{T}}, 
\]
with                probability               larger                than
$1   -  \mathcal{C}_2   \left  (   \frac{2  \sqrt{\xi}   \log(T)}{T}\vee
\frac{1}{T}\right   )$,    where   the    constants   $\mathcal{C}_0$,
$\mathcal{C}_1$ and $\mathcal{C}_2$ do not depend on $T$.

\section{Low-pass filter}
\label{sec:low-pass-filter}
In
this  section,  we consider  the  continuous-time  process described  in
Section~\ref{sec:continuous_noise} on the torus $\Theta = \R / \Z$ with 
$\lambda_T$   the Haar measure on $\Theta$, which is identified with the Lebesgue measure $\Leb$, and the noise:
\[
w_T=\sum_{k\in \N} \sqrt{\xi_k} \, G_k\,  \psi_k,
\]
where  $(G_k,  k\in  \N)$   are  independent  centered  Gaussian  random
variables with variance $\scale^2$, $\psi=(\psi_k, k \in \N)$ is an o.n.b. of
$L^2(\Leb)$ on
$\Theta$ and  $\xi=(\xi_k, k\in  \N)$ is a summable  sequence of
non-negative real numbers. The sequences $\psi$ and $\xi $ may depend on
$T$.   Recall from Section~\ref{sec:continuous_noise}
that the
noise satisfies  Assumption \ref{hyp:bruit}  for a positive  noise level
$\scale$     and      a     decay     on     the      noise     variance
$\Delta_T= \sup_{k \in \N} \xi_k $. 
\medskip

We consider the normalized Dirichlet kernel, see~\eqref{eq:low-pass-h},
on $\Theta$: 
\begin{equation}
\label{eq:low-pass-h-2}
\kernel(t,\sigma)  = \frac{\sin(T \pi  t)}{\sqrt{T} \, \sin(\pi t)}
\quad\text{for}\quad
t\in \Theta=\R/\Z \! \! \quad\text{and} \! \!\quad
\sigma = \frac{1}{T}, \!\! \quad T \in 2 \N^* +1.
\end{equation}
The parameter $T$ is related to the so-called cut-off frequency $f_c \in
\N^*$ by  $T=2f_c+1$.
We shall see below that the natural choice for the function $\Funk$
appearing in~\eqref{eq:def_K-app} is given by:
\begin{equation}
   \label{eq:def-funk-LPF}
\Funk(t)  =  \frac{\sin(\pi t)}{\pi t}
\quad\text{for}\quad t\in \R.
\end{equation}
We get from the definition
\eqref{eq:def-g}  that $g_\infty = - \Funk''(0)=\pi^2 / 3$.

\begin{remark}
  { Note that, if we consider the Shannon scaling
    function from multi-resolution approximation in
    Section~\ref{sec:features-spaces} with $\sigma_T=1/T$, then its
    kernel $\cK_T$ (see~\eqref{eq:def-KT}) 
    is exactly equal to $\ck_T^{\text{prox}}$ (see~\eqref{eq:def_K-app})
    with $\Funk$ from~\eqref{eq:def-funk-LPF}. 
Therefore there is no  approximation in this case. This example can be
treated similarly to the low-pass filter.}
  \end{remark}

In  the   following,  we  check  that   Assumption  \ref{hyp-estimation}
hold. Then, using Theorem \ref{maintheorem}, we
provide     a    prediction     bound    for     the    estimator     of
$(\beta^\star,  \vartheta^\star)$ solution  of the  optimization problem
\eqref{eq:generalized_lasso}.

\subsection{The approximating kernel}
We define
the    features    $\varphi_T$   using   \eqref{eq:def-features}    with
$\sigma_T = 1/T$.
Elementary calculations give that for $\theta, \theta'\in \Theta$:
\[
\ck_T (\theta, \theta') = \frac{\sin(T \pi  (\theta-\theta'))}{T \,
	\sin(\pi (\theta-\theta'))}
\cdot
\]

Recall that by convention $|\theta -\theta'|$ is the Euclidean distance
between $\theta$ and $ \theta'$ in $\Theta$, and in particular it
belongs to $ [0, 1/2]$.
We   define  the  approximating  kernel   $\ck_T^{\text{prox}}$ on
$\Theta$ by:
\[
\ck_T^{\text{prox}}(\theta, \theta')=\Funk(T|\theta-\theta'|)
\quad\text{with}\quad |\theta-\theta'|\in [0, 1/2].
\]
Since $\Funk$ is even, we get also that
$\Funk(T|\theta-\theta'|)=\Funk(T(\theta-\theta'))$ where, for
$\theta,  \theta'\in \Theta$,  their representers 
in $\R$ are chosen so that $\theta - \theta'$ belongs to $[-1/2, 1/2]$.

\subsection{Checking Assumption~\ref{hyp-estimation}}

\subsubsection{Regularity           of            the           dictionary}

It is elementary  to check  that 	$g_{T}$ is a constant function
on $\Theta$ equal to $ 
(T^2 - 1) \, g_\infty $ and that 
Assumption~\ref{hyp-estimation}~$\ref{hyp:theorem_properties_dico}$ on
the regularity of the dictionary 
holds. 
\subsubsection{Boundedness and local concavity on the diagonal}
There exists $R > 0$ such that for any $r \in  (0,R)$:
\[
\varepsilon (r) = 1 - \frac{\sin(\pi r)}{\pi r}> 0
\quad\text{and} \quad
\nu(r) = - \left (\frac{6}{\pi^3 r^3} - \frac{3}{\pi \,r}\right ) \sin( \pi r) + \frac{6\cos(\pi r)}{\pi^2 r^2} > 0.
\]
We fix $r \in (0,(1/ \sqrt{2 g_\infty L_2}) \wedge (R/2))$. This and the
fact that $F$ is  $\cc^\infty $ with bounded derivatives implies
that Assumption \ref{hyp-estimation}~$\ref{hyp:theorem_properties_k}$ on
the  boundedness and  the local  concavity of  the approximating  kernel
holds.
\subsubsection{Proximity to the approximating kernel}
We set  $\Theta_T = \Theta$. The proof of the next lemma on the  uniform
approximation of $\ck_T$ by $\ck_T^{\text{prox}}$ on the torus is
postponed to Section~\ref{sec:proof-app-Dir}.
\begin{lem}
	\label{lem:low-pass-prox}
	There exists a universal positive
	finite constant $c_3$ such that for any  $T \in  2\N^*+ 1$:
	\begin{equation}
	\label{eq:low-pass-rates}
	\DT \leq \frac{c_3}{T} \quad \text{and} \quad |1- \CT| \leq
	\frac{1}{2(T^2-1)}\cdot
	\end{equation}
\end{lem}

Let    $\eta     \in    (0,1)$    be    fixed.      We    deduce
from~\eqref{eq:low-pass-rates}
that              Assumption
\ref{hyp-estimation}~$\ref{hyp:proximity-setting}$ is  satified provided
that $T$  is larger  than some  constant depending  on $\eta$,  $r$, and
the
sparsity    $s$.

\subsubsection{Separation of the non-linear parameters}
Notice                                                              that
$\lim_{r''\rightarrow \infty }  \sup_{|r'|\geq r''} |\Funk^{(i)}(r')|=0$
for   all  $i\in   \{0,  \cdots,   3\}$.  Thus,   we  deduce   from  the
definition~\eqref{eq:def-delta-rs}  of  $\delta$ that  $\delta(u,s)$  is
finite for all $s\in \N^*$ and $u>0$.  \medskip

So Assumption \ref{hyp-estimation}~$\ref{hyp:separation}$ is verified as
soon as  $|\theta - \theta' |> \sigma_T \, \Sep$ for all  $\theta
\neq \theta' \in \cq^\star$. (Notice this happens for $T$ large enough
depending on $\cq^\star$ as $\sigma_T=1/T$.)

\subsection{Prediction error bound}
There exists a constant $c$ depending on $\eta$ and $r$ such that for
any $T\in 2\N^*+1$ such that $T \geq c \,s$, and provided
that~\eqref{eq:separation-Q-star} is satisfied,  Assumption
\ref{hyp-estimation} holds. Using Theorem \ref{maintheorem} with
$\kappa$ given by an equality in~\eqref{eq:bound-kappa} with $\tau > 1$,
we obtain  the prediction bound:
\begin{equation*}
\norm{\hat{\beta}\Phi_{T}(\hat{\vartheta}) -
	\beta^{\star}\Phi_{T}(\vartheta^{\star}) }_{L^2(\Leb)} \leq  \mathcal{C}_0 \,\mathcal{C}_1 \, \scale \, \sqrt{s\, \Delta_T\, \log (\tau)},
\end{equation*}
with probability larger than $1  -
\mathcal{C}_2 \left (  \frac{T}{ \tau \sqrt{\log (\tau)} }\vee
\frac{1}{\tau}\right )$, where the constants $\mathcal{C}_0$,
$\mathcal{C}_1$ and $\mathcal{C}_2$ do not depend on $T$.

\begin{remark}
	Exact     support    recovery     results     were    obtained     in
	\cite{duval2015exact}.   The authors  considered  a small  noise
	regime,  that  is:
        \begin{equation}
          \label{eq:petit-bruit}
          \norm{w_T}_{L^2(\Leb)} \leq  C \kappa,
        \end{equation}
        for some  finite constant  $C$. They  assumed that  the location
        parameters  satisfy for  any distinct $k  ,\ell \in  \{1,\cdots,s\}$,    the     separation    condition
        $|\theta^\star_k-\theta_\ell^\star|   \geq   C    /f_c   $   for
        $T=2f_c+1$, for some positive constant $C$ and with $f_c \geq s$
        ($s$ being the  number of active features in  the mixture). They
        showed that  there exist  finite constants  $C'$ and  $C''$ such
        that for all $k \in \{1,\cdots,s\}$:
	\[
	|\tilde \theta_k - \theta_k^\star| \leq 
	C'\norm{w_T}_{L^2(\Leb)} \quad \text{and} \quad
	|\tilde \beta_k-\beta^\star_k| \leq  C'' \norm{w_T}_{L^2(\Leb)} ,
	\]
	for        some        estimators       $(\tilde        \beta, \tilde
	\vartheta=(\tilde  \theta_1,\cdots,\tilde  \theta_s))$  obtained  by  solving  the
	BLasso  problem.
	
	However the small noise regime assumption is restrictive as it does not
	encompass the example of Section~\ref{sec:continuous_noise} where  for
	all $k \in \N$, $\xi_k = T^{-1}\ind_{ \{1 \leq k \leq T \}}$ and thus 
	$\Delta_T=1/T$ and $\mathbb{E}[\norm{w_T}_{L^2(\Leb)}] $ is of
        order $1$.  {So taking  $\kappa$ given
          by~\eqref{eq:bound-kappa} with an equality and $\tau=T$, we
          deduce that~\eqref{eq:petit-bruit} does not hold for $T$
          large.}
        Recall that in \eqref{eq:main_theorem_diff_l1} we obtain that
        our estimators satisfy: 
	\[
	\left |\|\,\hat{\beta}\|_{\ell_1} - \|\,\beta^\star \|_{\ell_1}
	\right | \leq C \frac {\sparse \, \sqrt{\log(T)}}{ \sqrt{T}}
	\]
	for some constant $C>0$ with high probability. Thus our prediction and estimation rates are smaller by a factor $\sqrt{\log(T)}/\sqrt{T}$ due to the probabilistic bounds on linear functionals of the noise process that we used in the proof, and this holds under an analogous separation condition on any $\theta_k^\star$ and $\theta_\ell^\star$, for $k \neq \ell$ in $\{1,...,\sparse\}$.
\end{remark}

%%%%%%%%%%%%%%%%%%%%%%%
\section{Technical proofs}
\label{sec:proofs}
%%%%%%%%%%%%%%%%%%%%%%%

%%%%%%%%%%%%%%%%%%%%%%%
\subsection{Proof of Lemma~\ref{lem:identifiability}}
\label{sec:proof_Lemma_2_4}

 First, for $s\geq 1$ and $\vartheta^\star = (\theta_1^\star, \cdots,\theta_s^\star)$ such that Assumption  \ref{hyp-estimation} stands for the set $\cq^\star$,  we show that the application $\beta \mapsto \beta \Phi_T(\vartheta^\star)$ defined from $\R^s$ to $L^2(\lambda_T)$ is injective.
	
	We have that $\norm{\beta\Phi_T(\vartheta^\star)}_{L^2(\lambda_T)} = \beta\Gamma \beta^\top $, where $\Gamma \in \R^{s \times s}$ is the symmetric matrix defined by $\Gamma_{k,\ell} = \cK_T(\theta_{k}^\star,\theta_{\ell}^\star)$.
	Let $\lambda_{\min}$ be the smallest eigenvalue of $\Gamma$. Using Gershgorin's theorem and the definition of $\DT$ given by \eqref{def:V_1}, we have that:
	\begin{multline*}
	\lambda_{\min} \geq 1  - \max_{1\leq  \ell \leq
		s} \sum\limits_{k=1, k\neq
		\ell}^{s} 
	|\cK_T(\theta_{\ell}^\star,\theta_{k}^\star)| \\\geq 1 - \max_{1\leq  \ell \leq
		s} \sum\limits_{k=1, k\neq
		\ell}^{s}
	\left |\Funk\left (\frac{|\theta_{\ell}^\star- \theta_{k}^\star|}{\sigma_T}\right )\right |  - (s-1)\DT.
	\end{multline*}
	The separation condition from Point $\ref{hyp:separation}$ of Assumption \ref{hyp-estimation} implies that for all $k,\ell \in \{1,\cdots,s\}$ such that $k \neq \ell$ we have $|\theta_k^\star - \theta_\ell^\star| \geq \sigma_T \Sep \geq 8 \, \sigma_T  \, \delta(\eta H_{\infty}^{(2)}(r),s)$.
	Recall the definition of $\delta(u,s)$ given by \eqref{eq:def-delta-rs}. We deduce that:
	\[
	\max_{1\leq  \ell \leq
		s} \sum\limits_{k=1, k\neq
		\ell}^{s}
	\left |\Funk\left (\frac{|\theta_{\ell}^\star- \theta_{k}^\star|}{\sigma_T}\right )  \right |  \leq \eta H_{\infty}^{(2)}(r).
	\]
	By Point $\ref{hyp:proximity-setting}$ of Assumption \ref{hyp-estimation}, we have $(s-1)\DT \leq (1-\eta) H_{\infty}^{(2)}(r) $ and $H_{\infty}^{(2)}(r) \leq 1/6$. Thus, we get:
	\begin{equation}
	\label{eq:small-eig}
	\lambda_{\min} \geq 5/6.
	\end{equation}
	Hence, the symmetric matrix $\Gamma$ is positive-definite.  This proves that the application $\beta \mapsto \beta \Phi_T(\vartheta^\star)$ is injective  from $\R^s$ to $L^2(\lambda_T)$. By symmetry, we obtain for $s^0 \geq 1$  that the application $\beta \mapsto \beta \Phi_T(\vartheta^0)$ is injective  from $\R^{s^0}$ to $L^2(\lambda_T)$.
	\medskip
	
	If $s=0$, we have $\beta^\star \Phi_T(\vartheta^\star) = 0$. For $s^0 \geq 1$, we have $\beta^0 \in (\R^*)^{s^0}$ and since $\beta \mapsto \beta \Phi_T(\vartheta^0)$ is injective, we deduce that $\beta^0 \Phi_T(\vartheta^0) \neq 0$. Thus, $s=0$ and $\beta^\star \Phi_T(\vartheta^\star) = \beta^0 \Phi_T(\vartheta^0)$ implies that $s^0=0$. By symmetry, $s^0= 0$ and $\beta^\star \Phi_T(\vartheta^\star) = \beta^0 \Phi_T(\vartheta^0)$ implies also that $s=0$.

	Assume from now on that $s,s^0 \in \N^*$ and that $\beta^\star \Phi_T(\vartheta^\star) = \beta^0 \Phi_T(\vartheta^0)$. Let us consider the application $v:\cq^\star \mapsto \{-1,1\}$ defined by: $v(\theta_k^\star) = \sgn (\beta_k^\star)$ for any $k \in \{1,\cdots,s\}$. According to Lemma \ref{lem:interpolating_certificate}, there exists $p^\star \in L^2(\lambda_T)$ such that:
	\[
	\norm{\beta^\star}_{\ell_1} = \sum_{k=1}^{s} \beta^\star_k\left \langle \phi_T(\theta^\star_k), p^\star\right \rangle_{L^2(\lambda_T)} = \left \langle \beta^\star\Phi_T(\vartheta^\star), p^\star\right \rangle_{L^2(\lambda_T)}.
	\]
	Using the fact that $\beta^\star \Phi_T(\vartheta^\star) = \beta^0 \Phi_T(\vartheta^0)$ and Properties $\ref{it:as1-<1}$ and $\ref{it:as1-<1-c}$ of $p^\star$ in Lemma \ref{lem:interpolating_certificate}, we get:
	\begin{equation}
	\label{eq:norm*<=norm0}
	\norm{\beta^\star}_{\ell_1}  = \sum_{k=1}^{s^0} \beta_k^0\left \langle \phi_T(\theta^0_k), p^\star\right \rangle_{L^2(\lambda_T)} \leq \norm{\beta^0}_{\ell_1}.
	\end{equation}
	The role of $(\beta^\star,\vartheta^\star)$ and $(\beta^0,\vartheta^0)$ being symmetric, we also get $\norm{\beta^0}_{\ell_1} \leq \norm{\beta^\star}_{\ell_1}$. Hence, we have $\norm{\beta^0}_{\ell_1} = \norm{\beta^\star}_{\ell_1}$ and  $\sgn (\beta^0_k) = \left \langle \phi_T(\theta^0_k), p^\star\right \rangle_{L^2(\lambda_T)}$ for $k \in \{1,\cdots,s^0\}$.
	Using Properties $\ref{it:as1-<1}$ and $\ref{it:as1-<1-c}$ of $p^\star$ in Lemma \ref{lem:interpolating_certificate}, we remark that for any $\theta \notin \cq^\star$
	\[
	\left|\left \langle \phi_T(\theta), p^\star\right \rangle_{L^2(\lambda_T)}\right| < 1.
	\]
	Thus, we deduce from \eqref{eq:norm*<=norm0} that $\cq^0 \subseteq \cq^\star$ and by symmetry $\cq^0 = \cq^\star$. Hence, we obtain $\vartheta^\star= \vartheta^0$ (up to a permutation on the components of $\vartheta^\star$) and $s=s^0$. Then use the injectivity of the function $\beta \mapsto \beta\Phi_T(\vartheta^\star)$ to get that $\beta^\star = \beta^0$ (up to the same permutation). This finishes the proof of the Lemma.
 
%%%%%%%%%%%%%%%%%%%%%%%%%%%%%%%%%%%%
\subsection{Proof of Theorem~\ref{th:non_sparse}}	
\label{sec:proof_thm_3_1}

 We give a bound of the type I error probability.	Using that under $H_0$ we have $y =  \beta^0\Phi_T(\vartheta^0) + w_T$, we get:
	\begin{equation*}
	\mathbb{E}_{(\beta^0,\, \vartheta^0)}[\Psi_{\Tau_1}(\threshold) ] =  \mathbb{P} \left ( \left  | \norm{w_T}_{L^2(\lambda_T)}^2 - \mathbb{E} \left  [ \norm{w_T}_{L^2(\lambda_T)}^2 \right ] \right | >\threshold  \right ).
	\end{equation*}
	Using Chebyshev's inequality,
	we obtain:
	\begin{equation}
	\label{eq:typeI-nonsparse}
	\mathbb{E}_{(\beta^0,\, \vartheta^0)}[\Psi_{\Tau_1}(\threshold) ]  \leq \frac{\varmodel}{\threshold^2}\cdot
	\end{equation}
	
	\medskip
	
	We now give a bound of the type II error probability. We set:
	\[
	R =\norm{\beta^0\Phi_T(\vartheta^0) - \beta^\star\Phi_T(\vartheta^\star)}_{L^2(\lambda_T)},
	\]	
	where $(\beta^\star, \vartheta^\star) \in (\R^*)^s \times \Theta_T^{s}(\delta^\star)$. Using the decomposition of $y$ from the model \eqref{eq:model} and the triangle inequality, we have: 
	
	\begin{multline*}
	|\Tau_1| \geq   R^2 -  \left |\norm{w_T}_{L^2(\lambda_T)}^2 - \mathbb{E}[\norm{w_T}_{L^2(\lambda_T)}^2] \right |\\
	 - 2 \left | \left \langle \beta^0\Phi_T(\vartheta^0) - \beta^\star\Phi_T(\vartheta^\star),w_T \right \rangle_{L^2(\lambda_T)} \right |.
	\end{multline*}
	Notice that by Assumption \ref{hyp:bruit}, the random variable $$\left \langle \beta^0\Phi_T(\vartheta^0) - \beta^\star\Phi_T(\vartheta^\star),w_T \right \rangle_{L^2(\lambda_T)},$$ is Gaussian with zero mean and variance bounded by $\scale^2 \,\Delta_T \, R^2$.
	Hence, using that under $H_1(\rho)$ we have $R \geq \rho$, we obtain:
	\begin{equation}
	\label{eq:typeII__decomp}
	\begin{aligned}
	\mathbb{E}_{(\beta^\star, \, \vartheta^\star)}[1-\Psi_{\Tau_1}(\threshold) ] \leq &\mathbb{P} \left ((\rho^2 - \threshold)/2  \leq  \left |\norm{w_T}_{L^2(\lambda_T)}^2 - \mathbb{E}[\norm{w_T}_{L^2(\lambda_T)}^2] \right | \right ) \\
	&+ \mathbb{P} \left ((R^2 - \threshold)/2  \leq  2\scale \sqrt{\Delta_T} \, R \,  | G | \right ),
	\end{aligned}
	\end{equation}
	where $G$ is a standard Gaussian random variable.
	On the one hand, for $\threshold < \rho^2$, using Chebyshev's inequality we get:
	\begin{equation}
	\label{eq:nonsparse_II_term_a}
	\mathbb{P} \left ((\rho^2 - \threshold)/2  \leq  \left |\norm{w_T}_{L^2(\lambda_T)}^2 - \mathbb{E}[\norm{w_T}_{L^2(\lambda_T)}^2] \right |\right ) \leq \frac{4\, \varmodel}{(\rho^2 - \threshold)^2} \cdot
	\end{equation}
	On the other hand, we have:
	\begin{equation}
	\mathbb{P} \left((R^2 - \threshold)/2   \leq 2\scale \sqrt{\Delta_T} \, R \,  | G | \right ) \! \leq \! \mathbb{P} \left ( \frac{ \rho^2 - \threshold}{4\scale \sqrt{\Delta_T}\rho}\leq  | G |\right )  
	\! \leq  \expp{ - ( \rho^2 - \threshold)^2/(32 \scale^2 \Delta_T \rho^2)}.\label{eq:nonsparse_II_term_b}
	\end{equation}
	where we used that $\rho \leq R$ and the tail bound  (see \cite[Formula~7.1.13]{abramowitz1964handbook}):
	\begin{equation}
	\label{eq:gaussian-tail}
	\frac{1}{\sqrt{2\pi}}\int_u^{+\infty}\expp{-t^2/2} \rd t\leq
	\, \frac{1}{2}\, \expp{-u^2/2}, \quad \text{  for } u>0.
	\end{equation}
	By combining \eqref{eq:typeII__decomp} with \eqref{eq:nonsparse_II_term_a} and \eqref{eq:nonsparse_II_term_b}, we get the following bound on the type II error probability:
	\begin{equation}
	\label{eq:typeII-nonsparse}
	\mathbb{E}_{(\beta^\star, \, \vartheta^\star)}[1-\Psi_{\Tau_1}(\threshold) ] \leq  \frac{4 \, \varmodel}{(\rho^2 - \threshold)^2} +  \expp{ - ( \rho^2 - \threshold)^2/(32\scale^2 \Delta_T \rho^2)}.
	\end{equation}
	Then, by putting together \eqref{eq:typeI-nonsparse} and \eqref{eq:typeII-nonsparse}, we obtain \eqref{eq:non_sparse}.

%%%%%%%%%%%%%%%%%%%%%%%%%%%%%%%
\subsection{Proof of Theorem \ref{th:higly_sparse}}
\label{sec:proof_thm_3_3}
	\textbf{Case $s > 0$.} Let  $(\beta^\star, \vartheta^\star) \in (\R^*)^s \times \Theta_T^{s}(\delta^\star)$. We  consider the estimators $(\hat \beta, \hat \vartheta)$ defined in \eqref{eq:generalized_lasso}.
	Notice that the hypotheses of Theorem \ref{maintheorem} are in force. We use the constants $\mathcal{C}_0$, $\mathcal{C}_1$, $\mathcal{C}_2$ defined therein. Under $H_0$, we have $s=s^0$. Thus, for  $\sqrt{\threshold} \geq \mathcal{C}_0 \, \sqrt{\sparse} \, \kappa$, we get the following bound on the type I error probability:
	\begin{equation}
	\label{eq:errorI_s>0}
	\begin{aligned}
	\mathbb{E}_{(\beta^0, \, \vartheta^0) }[\Psi_{\Tau_2}(\threshold) ] & \leq \mathbb{P}\left (  \norm{\hat{\beta}\Phi_T(\hat{\vartheta}) - \beta^\star\Phi_T(\vartheta^\star)}_{L^2(\lambda_T)} >\mathcal{C}_0 \, \sqrt{\sparse} \, \kappa \right ) \\&\leq \mathcal{C}_2 \left (  \frac{|\Theta_T|}{ \sigma_T  \, \tau \sqrt{\log (\tau)} }\vee
	\frac{1}{\tau}\right ),
	\end{aligned}
	\end{equation}
	where we used that $\beta^0 \Phi_T(\vartheta^0) = \beta^\star \Phi_T(\vartheta^\star)$ and that  $\sqrt{\threshold} \geq \mathcal{C}_0 \, \sqrt{\sparse} \, \kappa$ for the first inequality and Theorem \ref{maintheorem} for the second.
	
	\medskip
	
	We now bound the type II error probability. Under $H_1(\rho)$, since  $$\norm{\beta^\star \Phi_T(\vartheta^\star) - \beta^0 \Phi_T(\vartheta^0)}_{L^2(\lambda_T)} \geq \rho,$$  
	we obtain  that:
	\begin{equation}
	\label{eq:errorII_s>0}
	\begin{aligned}
		\mathbb{E}_{(\beta^\star, \, \vartheta^\star)}[1-\Psi_{\Tau_2}(\threshold) ] &\leq \mathbb{P}\left(\rho - \sqrt{\threshold} \leq \norm{\hat \beta \Phi_T(\hat \vartheta) - \beta^\star \Phi_T(\vartheta^\star)}_{L^2(\lambda_T)} \right) \\&\leq  \mathcal{C}_2 \left (  \frac{|\Theta_T|}{ \sigma_T  \, \tau \sqrt{\log (\tau)} }\vee
	\frac{1}{\tau}\right ),
	\end{aligned}
	\end{equation}
	where we used the triangle inequality for the first inequality and  Theorem \ref{maintheorem} as well as $\rho - \sqrt{\threshold} \geq \mathcal{C}_0 \, \sqrt{s} \, \kappa$ for the second. 
	
	\medskip
	
	\textbf{Case $s= 0$.} Since $s=0$, we have $y=w_T$ according to \eqref{eq:model}. Let us first bound the type I error probability $
	\mathbb{E}_{(\beta^0, \, \vartheta^0) }[\Psi_{\Tau_2}(\threshold) ] $. Assume that the hypothesis $H_0$ holds so that $s=s^0=0$. By definition we have:
	\begin{equation*}
	\mathbb{E}_{(\beta^0, \, \vartheta^0) }[\Psi_{\Tau_2}(\threshold) ]  =  \mathbb{P}\left ( \norm{\hat \beta \Phi_T(\hat \vartheta)}_{L^2(\lambda_T)}^2 > \threshold \right ).
	\end{equation*}
	We get from the definition of the estimators $\hat \beta$ and $\hat \vartheta$ from \eqref{eq:generalized_lasso} that:
	\begin{equation*}
	\frac{1}{2}\norm{w_T - \hat \beta \Phi_T(\hat \vartheta)}_{L^2(\lambda_T)}^2 + \kappa \norm{\hat \beta}_{\ell_1} \leq \frac{1}{2}\norm{w_T}_{L^2(\lambda_T)}^2.
	\end{equation*}
	By rearranging some terms in the equation above, we get:
	\begin{equation}
	\label{eq:decomp_T_s=0}
	\begin{aligned}
	\frac{1}{2}\norm{\hat \beta \Phi_T(\hat \vartheta)}_{L^2(\lambda_T)}^2  &\leq \left \langle \hat \beta \Phi_T(\hat \vartheta),w_T \right \rangle_{L^2(\lambda_T)} - \kappa \norm{\hat \beta}_{\ell_1} \\
	&\leq \norm{\hat \beta}_{\ell_1} \left(  \sup_{\Theta_T}| \left \langle \phi_T(\theta),w_T \right \rangle_{L^2(\lambda_T)}| - \kappa\right ).
	\end{aligned}
	\end{equation}
	Let us define the event:
	\begin{equation}
	\mathcal{A} = \{ \sup_{\theta \in \Theta_T}| \left \langle \phi_T(\theta),w_T \right \rangle_{L^2(\lambda_T)}| < \kappa\}.
	\end{equation}
	We deduce from \eqref{eq:decomp_T_s=0} that on the event $\mathcal{A}$ we have $\norm{\hat \beta \Phi_T(\hat \vartheta)}_{L^2(\lambda_T)} = 0$. Therefore we get:
	\begin{equation}
	\label{eq:errorI_s=0}
	\mathbb{E}_{(\beta^0, \, \vartheta^0) }[\Psi_{\Tau_2}(\threshold) ] \leq \mathbb{P}\left ( \norm{\hat \beta \Phi_T(\hat \vartheta)}_{L^2(\lambda_T)} > 0 \right) \leq \mathbb{P}(\mathcal{A}^c).
	\end{equation}
	We shall bound later $\mathbb{P}(\mathcal{A}^c)$, see \eqref{eq:bound_event_A_complementary}.
	
	\medskip
	
	We now consider  the type II error probability. We asume $H_1$, that is $$\norm{ \beta^0 \Phi_T(\vartheta^0)}_{L^2(\lambda_T)} \geq \rho.$$ We obtain:
	\begin{equation}
	\label{eq:errorII_s=0}
	\begin{aligned}
	\mathbb{E}_{(\beta^\star, \, \vartheta^\star)}[1-\Psi_{\Tau_2}(\threshold) ] &= \mathbb{P}\left  ( \norm{\hat \beta \Phi_T(\hat \vartheta) - \beta^0 \Phi_T(\vartheta^0)}_{L^2(\lambda_T)} \leq \sqrt{\threshold} \right ) \\&\leq \mathbb{P}\left  ( \rho - \sqrt{\threshold} \leq \norm{\hat \beta \Phi_T(\hat \vartheta))}_{L^2(\lambda_T)} \right ) \leq \mathbb{P}(\mathcal{A}^c).
	\end{aligned}
	\end{equation}
	where we used the definition of $\Tau_2$ and the triangle inequality for the first inequality, the second inequality of \eqref{eq:errorI_s=0}  as well as $\rho - \sqrt{\threshold} > 0$ for the second.
	
	We shall apply \cite[Lemma A.1] {butucea22} to bound $\mathbb{P}(\mathcal{A}^c)$. It amounts to controling the supremum of the Gaussian process $\theta \mapsto \left \langle \phi_T(\theta),w_T \right \rangle_{L^2(\lambda_T)}$. Recall that Assumptions~~\ref{hyp:reg-f}
	and~\ref{hyp:g>0} 
	hold. The function $\phi_T$ is of class $\cc^1$ from
	the interval $\Theta_T$ to $L^2(\lambda_T)$, with $\Theta_T$ a sub-interval of $\Theta$.  We have also, with $\phi_T^{[1]} = \tilde{D}_{1;\ck_T}[\phi_T]$, that:
	\begin{equation*}
	\norm{\phi_T(\theta)}_{L^2(\lambda_T)} = 1
	\quad \text{and} \quad
	\norm{\phi_T^{[1]}(\theta) }_{L^2(\lambda_T)}^2 = \cK_{T}^{[1,1]}(\theta,\theta) =  1.  
	\end{equation*}
	Since Assumption \ref{hyp:bruit} on the noise $w_T$ holds, the hypotheses of  \cite[Lemma A.1]{butucea22} hold and we deduce from \cite[Lemma A.1]{butucea22} (with $C_1=C_2 =1$ therein) that:
	\begin{equation*}
	\begin{aligned}
	\mathbb{P}(\mathcal{A}^c) &= \P \left ( \sup_{\theta\in \Theta_T}  |\left \langle \phi_T(\theta),w_T \right \rangle_{L^2(\lambda_T)}| \geq \kappa
	\right ) \\
	&\leq 3 \cdot \left ( \frac{2 \scale \sqrt{g_\infty}|\Theta_T|\sqrt{ \Delta_T}}{\sigma_T\kappa
	}\vee 1 \right )\,  \expp{-\kappa^2/(4 
		\scale^2 \Delta_T)},
	\end{aligned}
	\end{equation*}
	where  the diameter  $|\Theta_T|_{\mathfrak{d}_T}$ of the set $\Theta_T$ with respect to the metric $\mathfrak{d}_T$ is bounded by $2 \sqrt{g_\infty}|\Theta_T|/\sigma_T$ using \eqref{eq:equi-dT-dI} and the fact that $\CT \leq 2$.
	By taking $\kappa \geq 2 \scale \sqrt{\Delta_T \log ( \tau)}$, we get:
	\begin{equation}
	\label{eq:bound_event_A_complementary}
	\mathbb{P}(\mathcal{A}^c) = \P \left ( \sup_{\theta\in \Theta_T}  |\left \langle \phi_T(\theta),w_T \right \rangle_{L^2(\lambda_T)}| \geq \kappa
	\right ) \leq 3 \cdot \left ( \frac{\sqrt{g_\infty}|\Theta_T|}{\sigma_T\tau \sqrt{ \log (\tau)}
	}\vee \frac{1}{\tau} \right ).
	\end{equation}
	Notice that  the constant $\mathcal{C}_2$ from Theorem \ref{maintheorem} is equal to $2\sqrt{g_\infty} \, \mathcal{C}_2'$ where $\mathcal{C}_2'$ is given by \cite[$\mathcal{C}_2$  from Eq. (84) therein]{butucea22} and is greater than $3$.  The constant $\mathcal{C}_2$ depends only on $r$ and the function $\Funk$.
	Finally, by putting together \eqref{eq:errorI_s>0}, \eqref{eq:errorII_s>0}, \eqref{eq:errorI_s=0}  and \eqref{eq:errorII_s=0}, we obtain for $\kappa \geq \mathcal{C}_1\scale \sqrt{\Delta_T  \log (\tau)}$ (where the constant $\mathcal{C}_1$ is defined in  \cite[Proof of Theorem 2.1 (p.32)]{butucea22} and is superior to $4$) the bound on the maximal testing risk from Theorem \ref{th:higly_sparse}.
	This finishes the proof.

%%%%%%%%%%%%%%%%%%%%%%%%
\subsection{Proof of Theorem~\ref{theorem:inclusion}}
\label{sec:proof_thm_4_2}

This proof is based on the certificate function. 
Following \cite{butucea22}, we give  the existence and
properties of the interpolating
certificate function.

\begin{lem}[Interpolating certificate]
	\label{lem:interpolating_certificate}
	Let $T \in \N$, $s \in \N^*$, $\eta \in (0,1)$, $r \in \left (0,1/\sqrt{2 \, g_\infty \, L_{2}} \right )$ and $\cq = \{\theta_1,\cdots,\theta_s\} \subset \Theta_T$. Suppose  that   Assumption \ref{hyp-estimation} holds.
	
	Then, there exist finite positive constants $C_{N}  , C_{F} $,
	$C_{B}$  with $C_F < 1$, depending on $r$
	and the function $\Funk$, such that  for any application $v : \cq \mapsto \{-1,1\}$, there exist  unique $\alpha, \coeff \in \R^s$ such that $p \in L^2(\lambda_T)$ uniquely defined by:
	\begin{equation}
	\begin{cases}
	&p = \sum\limits_{k=1}^s \alpha_{k}
	\phi_{T}(\theta_{k})
	+ \sum\limits_{k=1}^{s} \coeff_{k} \, 
	\tilde D_{1,T}[\phi_{T}](\theta_{k}),\\
	&\langle \phi_{T}(\theta),p  \rangle_{L^2(\lambda_T)} =
	v(\theta) \quad \text{ and } \quad \left \langle \partial_\theta
	\phi_T(\theta),p \right \rangle_{L^2(\lambda_T)} = 0, \quad
	\text{for all} \quad \theta\in \cq,
	\end{cases}	
	\end{equation} 
	satisfies:
	\begin{propenum} 
		\item\label{it:as1-<1}
		For all $\theta \in \cq$ and $\theta'\in
		\mathcal{B}_{T}(\theta,r)$, we have:
		$$ | \langle \phi_{T}(\theta'),p  \rangle_{L^2(\lambda_T)}| \leq 1 -
		C_{N}\, \dT(\theta,\theta')^{2}.$$
		%\item\label{it:as1-ordre=2}
		%For all $\theta \in \cq$ and $\theta'\in
		%\mathcal{B}_T(\theta,r)$, we have
		%$ | \langle\phi_{T}(\theta'),p \rangle_T - v(\theta)| \leq
		%C_{N}' \, \mathfrak{d}_T(\theta,\theta')^{2}$.
		\item\label{it:as1-<1-c}  For all $\theta$ in $\Theta_T$, $\theta  \notin \bigcup\limits_{\theta'
			\in \cq} \mathcal{B}_{T}(\theta',r)$ (far region), we
		have: $$|\langle\phi_{T}(\theta),p \rangle_{L^2(\lambda_T)}| \leq 1 - C_{F}.$$
		\item \label{it:norm<c} We have $\norm{p}_{L^2(\lambda_T)} \leq \sqrt{s} \, C_{B}$.
	\end{propenum}
\end{lem}
\begin{proof}
	Using similar  arguments as those  developed in the proof  of Theorem
	\ref{maintheorem},    we   get    that   all    the   hypotheses    of
	\cite[Proposition,  7.4]{butucea22} are  satisfied. The  existence and
	uniqueness    of   $p$    is   then    guaranteed   by    \cite[Lemma,
	10.1]{butucea22}.  The   properties  satisfied   by  $p$   are  direct
	consequences of \cite[Proposition, 7.4]{butucea22}.
\end{proof}

	Recall the test problem given by \eqref{eq:H0-dico}. 	
	Assumption \ref{hyp-estimation} holds for the set $\cq^0$. Thanks to Lemma \ref{lem:interpolating_certificate}, the element $p_0$ of $L^2(\lambda_T)$ is uniquely defined by $v^0$, \eqref{eq:def_p0} and \eqref{matrix_system1}. Hence, the test statistic $\Tau_3$ from \eqref{def:test_statistic_dico}  is well-defined.
	
	\medskip
	
	We first bound the type I error probability. Let us fix  $(\beta^\star, \vartheta^\star) \in (\R^*)^s\times \Theta_T^{s}(\delta^\star)$ such that $H_0$ holds. Using that $y = \beta^\star\Phi_{T}(\vartheta^\star) + w_T$ and the triangle inequality, we obtain:
	\begin{align}
	\label{eq:lower-bound-T3}
	|\Tau_3| &= \left |\norm{\hat{\beta}}_{\ell_1} - \norm{\beta^\star}_{\ell_1} +  \norm{\beta^\star}_{\ell_1}-  \left \langle \beta^\star \Phi_{T}(\vartheta^\star),p_0 \right \rangle_{L^2(\lambda_T)}  -  \left \langle w_T,p_0 \right \rangle_{L^2(\lambda_T)}\right |\\ \nonumber
	&\leq \left |\norm{\hat{\beta}}_{\ell_1} - \norm{\beta^\star}_{\ell_1} \right |  +  |B| +  \left | \left \langle w_T,p_0 \right \rangle_{L^2(\lambda_T)}\right |, \nonumber
	\end{align}
	where:
	\begin{equation}
	\label{eq:def-B}
	B= \norm{\beta^\star}_{\ell_1} -  \left \langle \beta^\star
	\Phi_{T}(\vartheta^\star),p_0 \right \rangle_{L^2(\lambda_T)}.
	\end{equation}
	Since   $\cq^{\star,+} \subseteq \cq^{0, +}, \, \cq^{\star,-} \subseteq \cq^{0, -}$, we have   for all $k \in \{1,\cdots,s\}$:
	\[
	| \beta^\star_k| -  \left \langle \beta_k^\star
	\phi_{T}(\theta^\star_k),p_0 \right \rangle_{L^2(\lambda_T)}  
	= 0,
	\]
	we deduce that $B=0$ under $H_0$. 
	Hence, we have that:
	\begin{equation}
	\label{eq:decomp_errorI}
	\mathbb{E}_{(\beta^\star, \, \vartheta^\star)}[\Psi_{\Tau_3}(\threshold) ] \leq   \mathbb{P} \left ( \left |\norm{\hat{\beta}}_{\ell_1} - \norm{\beta^\star}_{\ell_1}\right |    > \threshold/2 \right ) + \mathbb{P}\left ( \left | \left \langle w_T,p_0 \right \rangle_{L^2(\lambda_T)}\right | > \threshold/2 \right ) . 
	\end{equation}
	Recall that under $H_0$, we have $s \leq s^0$. Therefore, since $\mathcal{C}_3 \, \kappa \, s^0 \leq \threshold /2$, we have $\mathcal{C}_3 \, \kappa \, s \leq \threshold /2$. We get from Theorem \ref{maintheorem} that:
	\begin{equation}
	\label{eq:majo-theorem}
	\mathbb{P} \left (\left |\norm{\hat{\beta}}_{\ell_1} - \norm{\beta^\star}_{\ell_1}\right | >  \threshold/2 \right ) \leq \mathcal{C}_2 \left (  \frac{|\Theta_T|}{\sigma_T \tau \sqrt{\log (\tau)} }\vee
	\frac{1}{\tau}\right ) .
	\end{equation}
	Then, thanks to Assumptions~\ref{hyp:bruit} and Lemma
	\ref{lem:interpolating_certificate}, the quantity  $\left
	\langle w_T,p_0 \right \rangle_{L^2(\lambda_T)}$ is a centered
	Gaussian random variable of variance bounded by $\scale^2
	C_B^2\Delta_T  s_0 $ where $C_B$ is the finite positive constant
	from Lemma \ref{lem:interpolating_certificate}. Hence we have,
	provided that $\threshold \geq 2\mathcal{C}_3 \, \kappa \, s^0 $
	with $\kappa \geq \mathcal{C}_1 \scale \sqrt{\Delta_T  \log
		(\tau)} $, that is, $t^2\geq (2 \mathcal{C}_1\mathcal{C}_3
	\scale s_0)^2\Delta_T  \log
	(\tau)$:
	\begin{equation*}
	\begin{aligned}
	\P\left ( \left \langle w_T,p_0 \right \rangle_{L^2(\lambda_T)} > \threshold / 2\right ) &\leq 
	\int_{\threshold/2}^{+\infty}\frac{\expp{-x^2/(2
			\scale^2 \Delta_T C_B^2 s_0 )}}{\sqrt{2\pi
			\scale^2\Delta_T C_B^2 s_0}}  \, \rd x
	\\
	&\leq   \frac{1}{2} \expp{-\frac{\threshold^2}{8 (\scale^2\Delta_T C_B^2 s_0)}} \leq  \frac{1}{2\tau^{a s_0}},
	\end{aligned}
	\end{equation*}
	with                   $a                  =                   (
	\mathcal{C}_1\mathcal{C}_3/C_B)^2/2$ and where  we used the tail
	bound \eqref{eq:gaussian-tail}.  It gives by symmetry that:
	\begin{equation}
	\label{eq:majo-gaussian}
	\P\left (| \left \langle w_T,p_0 \right \rangle_{L^2(\lambda_T)}| > \threshold / 2\right ) \leq  \frac{1}{\tau^{a \, s_0}}.
	\end{equation}
	Plugging \eqref{eq:majo-theorem} and \eqref{eq:majo-gaussian} in \eqref{eq:decomp_errorI}, we get:
	\begin{equation}
	\label{eq:errorI_inclusion}
	\sup_{(\beta^\star, \, \vartheta^\star) \in H_0} \mathbb{E}_{(\beta^\star, \, \vartheta^\star) } \left [ \Psi_{\Tau_3}(\threshold) \right] \leq \mathcal{C}_2 \left (  \frac{|\Theta_T|}{ \sigma_T \,\tau \, \sqrt{\log (\tau)} }\vee
	\frac{1}{\tau}\right ) +   \frac{1}{\tau^{a \, s_0}}\cdot
	\end{equation}
	
	\medskip
	
	We now bound the type II error probability. Assume that $H_1$ holds, that is $\mathcal{D}_{T,r}(\beta^\star,\vartheta^\star,v^0, \vartheta^0) \geq\rho$.  We have, using the first equality of \eqref{eq:lower-bound-T3} and the triangle inequality, that:
	\begin{equation*}
	\begin{aligned}
	|\Tau_3| &\geq    \left | B \right |  -  \left | \left \langle w_T,p_0 \right \rangle_{L^2(\lambda_T)}\right | - \left |\norm{\hat{\beta}}_{\ell_1} - \norm{\beta^\star}_{\ell_1}\right |,
	\end{aligned}		
	\end{equation*}
	with $B$ defined in \eqref{eq:def-B}.
	Using  the definitions \eqref{def:sets_S0} of  $S(r)$ and
	$S_{k}^\epsilon(r)$ with $\epsilon\in \{+,-\}$ and $k \in \mathcal{I}^\epsilon$, we get:
	\begin{multline*}
		B =   \sum_{\substack{\epsilon\in \{+,-\}\\ k \in
			\mathcal{I}^\epsilon, \,  \ell \in S_k^\epsilon(r)}}  |\beta^\star_{\ell}| \left ( 1 -
	\sgn (\beta^\star_\ell) \, \left \langle \phi_{T}(\theta^\star_{\ell}) ,p_{0} \right
	\rangle_{L^2(\lambda_T)} \right ) \\
	+ \sum\limits_{k \in S(r)^{c}} \! \!
	|\beta^\star_{k}| \left ( 1 -    \sgn (\beta^\star_k) \, \left \langle
	\phi_{T}(\theta^\star_{k}) ,p_{0} \right \rangle_{L^2(\lambda_T)}  \right).
	\end{multline*}
	Thanks to Lemma~\ref{lem:interpolating_certificate}~$\ref{it:as1-<1}$-$\ref{it:as1-<1-c}$ of ,
	we obtain:
	\begin{equation*}
	\begin{aligned}
	B  &\geq \sum_{\substack{\epsilon\in \{+,-\}\\ k \in
			\mathcal{I}^\epsilon, \,  \ell \in S_k^\epsilon(r)}}  C_{N} | \beta_{\ell}^\star| \dT
	(\theta_{\ell}^\star, \theta_{k}^{0})^{2} + \sum\limits_{k \in
		S(r)^{c}} C_{F}|\beta_{k}^\star|\\
	& \geq (C_N \wedge C_F) \mathcal{D}_{T,r}(\beta^\star,\vartheta^\star,v^0, \vartheta^0)  \geq (C_N  \wedge  C_F) \rho,
	\end{aligned}
	\end{equation*}
	where the constants $C_N$ and $C_F$ are defined in Lemma \ref{lem:interpolating_certificate} and depend on $r$ and on the function $\Funk$.
	Therefore, we have with $a_\threshold=(C_N \wedge C_F) \rho -\threshold$:
	\begin{multline*}
	\mathbb{E}_{(\beta^\star,\vartheta^\star)} \left [ 1 -
	\Psi_{\Tau_3}(\threshold)\right]
	\leq 	\mathbb{P}\left ( \left | \left \langle w_T,p_0 \right
	\rangle_{L^2(\lambda_T)}\right | + \left
	|\norm{\beta^\star}_{\ell_1} - \norm{\hat{\beta}}_{\ell_1}\right |
	\geq a_\threshold \right )\\
	 \leq \mathbb{P}\left (   \left | \left \langle w_T,p_0 \right
	\rangle_{L^2(\lambda_T)}\right |  \geq
	a_\threshold/2 \right) \\+	\mathbb{P}\left ( \left
	|\norm{\beta^\star}_{\ell_1} - \norm{\hat{\beta}}_{\ell_1}
	\right | \geq  a_\threshold/2\right ). 
	\end{multline*}
	Provided that $\rho \geq \mathcal{C}_4 \, s \, \kappa + \mathcal{C}_5 \, \threshold$ with $\mathcal{C}_4= 2 \, \mathcal{C}_3 /(C_N \wedge  C_F)$  and $\mathcal{C}_5=2/(C_N \wedge  C_F)$ we have $a_\threshold/2 \geq (\mathcal{C}_3 \kappa s) \vee (\threshold /2) $. By using \eqref{eq:majo-theorem} and \eqref{eq:majo-gaussian}, we obtain:
	\begin{equation}
	\label{eq:errorII_inclusion}
	\sup_{(\beta^\star, \, \vartheta^\star) \in H_1(\rho)} \mathbb{E}_{(\beta^\star, \, \vartheta^\star) } \left [ 1 - \Psi_{\Tau_3}(\threshold) \right] \leq  \mathcal{C}_2 \left (  \frac{|\Theta_T|}{ \sigma_T \,\tau \, \sqrt{\log (\tau)} }\vee
	\frac{1}{\tau}\right ) +   \frac{1}{\tau^{a s_0}}.
	\end{equation}
	Finally, by adding both sides of  \eqref{eq:errorI_inclusion} and \eqref{eq:errorII_inclusion}, we get \eqref{eq:errorI+II_inclusion}.
	This concludes the proof.

%%%%%%%%%%%%%%%%%%%%%%%

\subsection{Proof of Lemma~\ref{lem:low-pass-prox}}
\label{sec:proof-app-Dir}

It    is   easy    to   check    that   the    functions   $g_T$    and
$g_{\ck_T^\text{prox}}$ are constant functions with:
\begin{equation}
g_{T} = g_\infty \, (T^2 - 1) \quad \text{and} \quad
g_{\ck_T^\text{prox}} = g_\infty \, T^2. 
\end{equation}
Thus,     we    easily     deduce    the     second    inequality     of
\eqref{eq:low-pass-rates}  from  the  definition  \eqref{eq:def-rho}  of
$\CT$.

\medskip

We now consider the bound on $\DT$. 
For  $i,j \in \{0,\cdots,3\}$ and $\ell=i+j$, we have with $\alpha_T=1- 1/T^2$:
\begin{equation}
\label{eq:bound-KT-K-inf}
\sup_{\Theta^2} |\cK_T^{[i,j]} - \cK_T^{{\text{prox}}[i,j]}|
=
g_\infty^{-\ell/2} (T^2 \alpha_T)^{-\ell/2} A_{\ell,T} ,
\end{equation}
where
\begin{equation*}
A_{\ell,T} = \sup_{t\in
		[-\frac{1}{2},\frac{1}{2}]} \left |  \partial_t^{\ell}  \left [
	D_T(t) + \left (1- \alpha_T^{\ell/2}\right) \frac{\sin(T \pi t)}{T \pi t}\right ] \right |,
\end{equation*}
and, for $t\in [-1/2, 1/2]$ and the convention $J(0)=0$:
\[
D_T(t) = \frac{\sin(T \pi t)}{T} J(t)
\quad\text{and}\quad
J(t) = \frac{1}{\sin(\pi t)} - \frac{1}{\pi t}\cdot
\]
It  is   easy  to  check  that the function  $J$  can be expanded as a power series at
0 with positive  convergence radius, and thus is of class $\cc^\infty $
on $ [-1/2, 1/2]$. Thus the following
constant is finite:
\[
M = \sup_{0 \leq \ell \leq 6}\,\,  \sup_{ [-1/2, 1/2]} \,
|J^{(\ell)}|<+\infty .
\]
Using the Leibniz rule, we have that for $\ell \in \{1,\cdots,6\}$ and
$t\in [-1/2, 1/2]$:
\begin{equation*}
|\partial_t^{\ell} D_T(t)| = \frac{1}{T} \left |\sum_{j=0}^\ell \binom{\ell}{j}
(T\pi)^j \, \sin^{(j)}(T \pi t) \, J^{(\ell-j)}(t) \right | \leq M
\frac{(T\pi +1)^\ell}{T}\cdot
\end{equation*}
We deduce from~\eqref{eq:bound-KT-K-inf} that for  $i,j \in
\{0,\cdots,3\}$ and $\ell=i+j$:
\begin{equation*}
\begin{aligned}
\sup_{\Theta^2} |\cK_T^{[i,j]} - \cK_T^{{\text{prox}}[i,j]}|
&\leq  g_\infty^{-\ell/2} (T^2 \alpha_T)^{-\ell/2} \left(M
\frac{(T\pi +1)^\ell}{T} + (1- \alpha_T^{\ell/2})\right) \\
&\leq  M3^\ell
\, T^{-1}, 
\end{aligned}
\end{equation*}
where we used that $T\geq 3$ and $g_\infty \alpha_T\geq 1$, and that $1
-\alpha_T^{\ell/2}=0$ for $\ell=0$. 
Recall the definition \eqref{def:V_1} of $\DT$ to get $\DT\leq   M3^\ell
\, T^{-1}$. 
This finishes the proof.

% \begin{acks}[Acknowledgments]
% \end{acks}

%%%%%%%%%%%%%%%%%%%%%%%%%%%%%%%%%%%%%%%%%%%%%%
%% Supplementary Material, if any, should   %%
%% be provided in {supplement} environment  %%
%% with title and short description.        %%
%%%%%%%%%%%%%%%%%%%%%%%%%%%%%%%%%%%%%%%%%%%%%%
%\begin{supplement}
%\stitle{???}
%\sdescription{???.}
%\end{supplement}

%% if your bibliography is in bibtex format, uncomment commands:
\bibliographystyle{imsart-number} % Style BST file (imsart-number.bst or imsart-nameyear.bst)
\bibliography{ref}      % Bibliography file (usually '*.bib')

%% or include bibliography directly:
% \begin{thebibliography}{}
% \bibitem{b1}
% \end{thebibliography}

\end{document}